%% file: single_measurement_2010_11m_10d.tex
\documentclass[12pt]{article}
\usepackage{amsmath,amssymb,latexsym}
\usepackage{amsthm}
\usepackage{url}
\usepackage{graphicx}
\usepackage{color}
\newtheorem{lemma}{Lemma}

\newtheorem{theorem}{Theorem}
\newtheorem{definition}{Definition}

\newtheorem{corollary}{Corollary}
\newtheorem{remark}{Remark}
\def\C{\mathbb C}
\def\R{\mathbb R}

\def\Z{\mathbb Z}
\def\N{\mathbb N}
\def\I{\mathcal I}
\def\F{\mathcal F}
\def\E{\mathcal E}
\def\D{\mathcal D}
\def\S{\mathcal S}

\def\p{\partial}
\def\e{\epsilon}
\def\grad{\hbox{grad}}

\def\sup{\mathop{\rm sup}}
\def\Im{\mathop{\rm Im}}
\def\supp{\mathop{\rm supp}}

\def\Tr{{\rm Tr}}
\def\diam{\mathop{\rm diam}}
\def\l{\left}
\def\r{\right}
\newcommand{\pair}[1]{\left\langle #1 \right\rangle}
\newcommand{\norm}[1]{\left\|#1 \right\|}
\def\bar{\overline}
\def\possiblebar{}
\def \ba {\begin {eqnarray*} }
\def \ea {\end {eqnarray*} }
\def \beq {\begin {eqnarray}}
\def \eeq {\end {eqnarray}}

\begin{document}
\title{An inverse problem for the wave equation with one  %compressed
measurement and the pseudorandom noise}
\author{Tapio Helin$^1$, Matti Lassas$^{2}$ and Lauri Oksanen$^2$}
\date{Nov. 10, 2010} 
\footnotetext[1]{Johann Radon Institute for Computational and Applied Mathematics (RICAM),
Austrian Academy of Sciences,
Altenbergerstra{\ss}e 69,
A-4040 Linz, Austria.
Email: tapio.helin$@$oeaw.ac.at.
} 
\footnotetext[2]{Department of Mathematics and Statistics, 
P.O. Box 68 (Gustaf H\"allstr\"omin katu 2b)
FI-00014 University of Helsinki.
Email: matti.lassas$@$helsinki.fi,
lauri.oksanen$@$helsinki.fi.
} 
\maketitle

{\bf Abstract.} {\it 
We consider the wave equation $(\p_t^2-\Delta_g)u(t,x)=f(t,x)$, in $\R^n$, $u|_{\R_-\times \R^n}=0$, 
where the metric $g=(g_{jk}(x))_{j,k=1}^n$ is known outside an open and bounded set $M\subset \R^n$
with smooth boundary $\p M$. 
We define a deterministic source $f(t,x)$ called the pseudorandom noise as a sum of point sources, 
$f(t,x)=\sum_{j=1}^\infty a_j\delta_{x_j}(x)\delta(t)$,
where the points $x_j,\ j\in\Z_+$, form a dense set on $\p M$. 
We show that when the weights $a_j$ are chosen
appropriately, $u|_{\R\times \p M}$ determines the scattering relation on $\p M$,
that is, it determines for all geodesics which pass through $M$
the travel times together with the entering and exit points and directions.
The wave $u(t,x)$ contains the singularities produced by all point sources,
but when $a_j=\lambda^{-\lambda^{j}}$ for some $\lambda>1$,
we can 
%in each measurement point 
trace back the point source that produced a given singularity in the data.
This gives us the distance in $(\R^n, g)$ between a source point $x_j$ and an arbitrary point $y \in \p M$. 
In particular, if $(\bar M,g)$ is a simple Riemannian manifold and $g$ is conformally Euclidian in $\bar M$, 
these distances are known to determine the metric $g$ in $M$. 
In the case when $(\bar M,g)$ is non-simple we present a more detailed analysis of the wave fronts 
yielding the scattering relation on $\p M$.
}

%\vspace{1cm}
%\tableofcontents

\section{Introduction}

In this paper we consider an inverse problem for the wave equation
\begin{align*} %\label{eq:wave_eq_1}
&(\p_t^2 - \Delta_g) u(t,x) = f(t,x) 
\quad {\rm in} \; (0, \infty) \times \R^n, 
\\&u|_{t=0}=\p_t u|_{t=0}=0, \nonumber
\end{align*}
where $\Delta_g$ is 
the Laplace-Beltrami operator corresponding to a Riemannian
metric $g(x)=[g_{jk}(x)]_{j,k=1}^n$, that is
\begin{equation*}
\Delta_g u=\sum_{j,k=1}^n |g|^{-1/2}\frac \p{\p x^j}\left( |g|^{1/2}g^{jk}\frac \p{\p x^k} u\right),
\end{equation*}
where $|g|=\det (g_{jk})$ and $[g^{jk}]_{j,k=1}^n=g(x)^{-1}$ is the inverse matrix of $[g_{jk}(x)]_{j,k=1}^n$.

We assume that $g_{jk} \in C^\infty (\R^n)$
and that there are $c_1,c_2>0$ such that
\begin{equation} 
\label{assumption_1}
c_1 |\xi|^2 
\leq \sum_{j.k=1}^n g_{jk}(x) \xi^j\xi^k 
\leq c_2 |\xi|^2,
\quad x, \xi \in \R^n.
\end{equation}
Moreover, we assume that the metric $g$ is known outside an
open and bounded set $M \subset \R^n$ 
having a $C^\infty$ smooth boundary $\p M$.

Denote by $d_{\possiblebar M}(x,y)=d_{\possiblebar M,g}(x,y)$, $x, y \in \bar M$, 
the distance function of Riemannian manifold 
$(\bar M, g)$, where $g$ is considered as its restriction to $\bar M$.
Let $T>\diam(M)$, where $\diam(M) = \max \{d_{\possiblebar M}(x,y) ; x,y \in \bar M\}$.

We choose the origin of the time axis so that the source $f$ is active at time $t=0$.
To ensure compatibility with the the initial conditions we let $T_0 < 0$ and 
define the measurement map $L=L_g$,
\begin{equation}
\label{L_map}
L : C^\infty_c(T_0, T) \otimes C_c^\infty(\R^n)
\to C^\infty( (T_0, T) \times \p M),\quad Lf= u|_{(T_0,T) \times \p M},
\end{equation}
where $u$ is the solution of the wave equation 
\begin{align} \label{eq:wave_eq_1}
&(\p_t^2 - \Delta_g) u(t,x) = f(t,x) 
\quad {\rm in} \; (T_0, T) \times \R^n, 
\\&u|_{t=T_0}=\p_t u|_{t=T_0}=0. \nonumber
\end{align}
Above, $C^\infty_c(T_0, T)$  denotes the space of
smooth functions having compact support in 
$(T_0, T)$. Its  dual space, the space of generalized functions or distributions,
is denoted by ${\cal D}'(T_0,T)$. Moreover, for functions
$\phi\in C^\infty_c(T_0, T)$ and $\psi\in C_c^\infty(\R^n)$ we denote
their pointwise product by $(\phi\otimes \psi)(t,x)=\phi(t)\psi(x)$.  

We remark that the assumption (\ref{assumption_1}) together with the finite speed of propagation for the wave equation
imply that the measurement $Lf$ does not depend on $g_{jk}(x)$, for $|x| > R$, when $R$ is sufficiently large. 
Thus we may assume without loss of generality
that all the partial derivatives $\p_x^\alpha g_{jk}$ are bounded on $\R^n$.

Let $x_j\in \p M$, $j = 1, 2, \dots$,
be a dense sequence of points in $\p M$, and
let us consider point sources
\begin{equation*}
f_{x_j}(t,x) := \delta(t) \delta_{x_j}(x), \quad j = 1, 2, \dots.
\end{equation*}
In order to study the measurements $Lf_{x_j}$,
we will use Sobolev spaces, see \cite{TRI}, 
\begin{align*}
H^s_p(\R^d) &:= \{f \in \S'(\R^d); \norm{f}_{H^s_p(\R^d)} := 
    \norm{(1 - \Delta)^{s/2} f}_{L^p(\R^d)} < +\infty \},
\\ \widetilde H^s_p(U) &:= \{f \in H^s_p(\R^d) ; \supp f \subset \bar U \},
\\ H^s_p(U) &:= \{f \in \D'(U); \text{$f = h|_U$ for some $h \in H^s_p(\R^d)$} \},
\end{align*}
where $U \subset \R^d$ is open and $s \in \R$. When $p = 2$ we omit the subscript $p$ in our notation,
that is, we denote $H^s(U)=H^s_2(U)$ etc. 
Moreover, we use projective topology on the tensor product $X\otimes Y$
of two Banach $X$ and $Y$, that is,
$\|z\|_{X \otimes Y} := \inf \sum_j \|x_j\|_X  \|y_j\|_Y,$ where
infimum is taken over all representations $z=  \sum_j x_j\otimes y_j$. We also
use projective topology on tensor products of locally convex spaces,
 see e.g.\ \cite[Def.\ 43.2]{Treves}. 
The measurement $Lf_{x_j}$ can be defined in the sense of the following lemma.

%As discussed in Appendix, 
%the values of $L f$ for all 
%$f \in C^\infty_0((\R^n \setminus M) \times (T_0, T))$
%and the metric $g_{jk}$ in $\R^n \setminus M$ 
%determine the hyperbolic Neumann-to-Dirichlet operator 
%of the wave equation in $M \times (T_0,T)$. 
%\note{Should we assume $T > 2 \diam(M)$ here?}
%Neumann-to-Dirichlet operator determines the metric in $M$ uniquely up to an isometry. 

\begin{lemma}
\label{lem:extending_measurement_map}
%Let $B \subset \R^n$ be a bounded domain with smooth boundary,
Let $p \in (1, \frac{n}{n-1})$ and 
let $m \in \N$ satisfy $m > \frac{n + 1}{4}$.
Then the measurement operator $L$, defined in (\ref{L_map})
%\begin{equation*}
%L : C^\infty_0(T_0, T) \otimes C_0^\infty(\R^n)
%\to C^\infty( (T_0, T) \times \p M)
%\end{equation*}
has a unique continuous extension 
\begin{equation*}
L:\widetilde H^{-1} (T_0,T) \otimes H^{-1}_p (\R^n) 
\to \D'((T_0,T) \times \p M).
%\to H^{-2m - 1} ((T_0,T); L^p(\p M)).
\end{equation*}
%In particular, $\Tr|_{\p M} W f \in \D'((T_0,T) \times \p M)$ when 
%$f$ is the pseudorandom noise TODO ref.
\end{lemma}

We will prove Lemma \ref{lem:extending_measurement_map} and
other results presented in introduction in Sections \ref{sec:measurement_map}-\ref{sec:travel_times}.

%By invariance of the wave equation with respect to time translation
%we can find $L f$ for all point sources 
%$f = \delta(t-t_0)\delta(x-x_0)$, $x_0\in \p M$ and $t_0 \in (T_0,T)$.
%Hence, the measurements $L f_{x_j}$, $j = 1, 2, \dots$,
%determine the Schwartz kernel of the map 
%\begin{equation*}
%L :  C^\infty_0( (T_0, T) \times \p M) \to \D'((T_0,T) \times \p M).
%\end{equation*}
%\note{TODO notice that we do not know $L$ on 
%$C^\infty_0((\R^n \setminus M) \times (T_0, T))$.
%Previously it was said that $L^T$ is determined.}
%As discussed in Appendix, 
%this map determines the metric in $M$ uniquely up to an isometry.
%
In this paper we study
a single measurement $L h_0$ that simultaneously combines all the measurements 
$L f_{x_j}$ by adding those together with appropriate weights.
%, $j = 1, 2, \dots$.
%corresponding to the point sources at points $x_j\in \p M$, $j=1,2,\dots$ 
%with appropriate weights 
%into a single simultaneous measurement.
%on a finite time interval. 
When the measurements $L f_{x_j}$ are summed together,
to the authors knowledge, there are no algorithms which could filter 
the value of a particular measurement from the sum.% of such measurements.
We will ask, however, can we find 
the essential features given by these measurements, 
like the travel times between points on $\p M$, so
that the metric could be determined under certain geometric conditions.

In Section \ref{sec:pseudo_white_noise}, Definition \ref{def:pseudorandom_noise}, 
we construct a specific function $h_0(t,x)$, called {\it pseudorandom noise}, 
so that $Lh_0$ determines the {\it scattering relation} $\Sigma_{M,g}$ for
the manifold $(\bar M,g)$. 
The scattering relation has been efficiently used to solve several 
geometric inverse problems \cite{lens2,lens3,lens1,lens4}.  

To define the scattering relation, let $TM$ 
denote the tangent space of $M$ 
and let $\dot \gamma$ denote the tangent vector of a smooth curve $\gamma : [a, b] \to M$.
Let $SM=\{(x,\xi)\in TM;\ \|\xi\|_g=1\}$ be the unit sphere bundle on $M$ and define
\begin{equation*}
\p_\pm SM=\{(x,\xi)\in SM;\ x\in \p M,\ \mp(\nu,\xi)_g>0\}, 
\end{equation*}
where $\nu$ is the exterior normal vector of $\p M$. 
Moreover, let $\tau(x,\xi)$ be the infimum of the set 
$\{ t\in (0, \infty] ; \gamma_{x,\xi}(t) \in  \p M \}$,
where $\gamma_{x,\xi}$ denotes the geodesic with initial data $(x,\xi)\in TM$.
We define the infimum of empty set to be $+\infty$.

The scattering relation is the map $\Sigma=\Sigma_{M,g}$,
\begin{equation*}
\Sigma:\D(\Sigma)\to \overline{\p_+SM} \times \R ,\quad \D(\Sigma)=\{(x,\xi)\in \p_-SM;\ \tau(x,\xi)<\infty\}
\end{equation*}
defined by $\Sigma(x,\xi)=(\gamma_{x,\xi}(\tau(x,\xi)),\dot \gamma_{x,\xi}(\tau(x,\xi)), \tau(x,\xi))$.

Our main result is the following.
\begin{theorem}
\label{main_thm} 
Let $M \subset \R^n$, $n\geq 2$ be an open and bounded set having a $C^\infty$ smooth boundary.
Then there is a generalized function $h_0(t,x)$ such that it is supported on 
$\{0\} \times \p M$ and has the following properties:
Assume that  $g_{jk},g'_{jk}  \in C^\infty (\R^n)$ are two Riemannian metric tensors
satisfying (\ref{assumption_1}). 
Moreover, assume that $g_{jk}(x)=g'_{jk}(x) $ for $x\in \R^n\setminus M$.
Let $T>\max(\diam(M,g), \diam(M,g'))$, $T_0<0$, and assume that
\begin{equation*}
L_g h_0 = L_{g'} h_0 \quad \hbox{on }(T_0, T) \times \p M.
\end{equation*}
Then the scattering relations $\Sigma_{M,g}$ and $\Sigma_{M',g'}$ of Riemannian manifolds 
$(M,g)$ and  $(M,g')$ are the same.  In particular,
if $(\bar M,g)$ and $(\bar M,g')$ are simple, 
the restrictions of the distance functions on the boundary satisfy
$d_{\possiblebar M,g}(x,y)=d_{\possiblebar M,g'}(x,y)$ for $x,y\in \p M$.
\end{theorem}

Recall that a compact Riemannian manifold $(\bar M, g)$ 
with boundary is simple
if it is simply connected, any geodesic has no conjugate points and
$\p M$ is strictly convex with respect to the metric $g$.
Any two points of a simple manifold can be joined by
a unique geodesic.

%Roughly speaking, Theorem \ref{main_thm} means that there is a boundary source
%$h_0$, not depending on metric $g$, such that if $(M,g)$ is simple and $g$ is known
%outside $M$, then
%the  measurement $L_g h_0=u|_{(T_0, T) \times \p M}$ determines uniquely 
%the distances $d_{\possiblebar M}(x,y)$ for $x,y\in \p M$.

The key idea of proof of Theorem \ref{main_thm} is to use source
$h_0(t,x)=\sum_{j=1}^\infty a_j f_{x_j}$. 
The point source $a_{j_0} f_{x_{j_0}}$ produces a singularity, 
which is observed at a point $y\in \R^n\setminus M$ at time $t_0=d(x_{j_0},y)$ 
with a magnitude $a_{j_0}\beta(x_{j_0},y)$, where $\beta$ is an unknown nonvanishing smooth function. 
Appropriate choice of the weights $a_j$ allows us find the index $j_0$ by 
looking nearby singularities. 
Indeed, when $x_{j_k} \to x_{j_0}$ and $j_k \to \infty$, 
we see that the asymptotic behavior of the magnitude $a_{j_k}\beta(x_{j_k},y)$ as $k \to \infty$
will be that of the weights $a_{j_k}$. 
Thus it is possible to factor out $a_{j_k}$ in the magnitude and determine $a_{j_0}$.
This argument is presented in Section \ref{sec:scattering_relation} and 
gives us the distances $d(x_j,y)$ in $(\R^{n}, g)$ for
arbitrary point $y\in \R^n\setminus M$ and a source point $x_j$.

Theorem \ref{main_thm} and boundary rigidity results for simple manifolds imply the following:

\begin{corollary}
\label{cor: main_thm} 
Let $M \subset \R^n$ and let $g_{jk},g'_{jk}  \in C^\infty (\R^n)$ be two Riemannian metric tensors
satisfying assumptions of Theorem \ref{main_thm}. 
Assume that  $(\overline M,g)$ and $(\overline M,g')$ are simple Riemannian manifolds.
Then 
\medskip

(i) If $n=2$ and 
\begin{equation}\label{coincide}
L_g h_0 = L_{g'} h_0 \quad \hbox{on }(T_0, T) \times \p M
\end{equation}
then there is a diffeomorphism $\Phi:M\to M$ such that $\Phi|_{\p M}=Id$
and $g=\Phi_*g'$.

(ii) For $n\geq 3$ there is $\epsilon=\e_{n,M}>0$ such that if 
$\|g_{jk}-\delta_{jk}\|_{C^2(M)} <\epsilon_n$, $\|g'_{jk}-\delta_{jk}\|_{C^2(M)}<\epsilon_n$
and (\ref{coincide}) holds,
then there is a diffeomorphism $\Phi:M\to M$ such that $\Phi|_{\p M}=Id$
and $g=\Phi_*g'$.

(iii) If $g_{jk}(x)=a(x)\delta_{jk}$ and 
$g'_{jk}(x)=a'(x)\delta_{jk}$, that is, the metric tensors are conformally Euclidian,
and (\ref{coincide}) holds, then $g_{jk}(x)=g_{jk}'(x)$ for $x\in M$.
\end{corollary}

Indeed, by Theorem \ref{main_thm}, the case (i) follows from \cite{PU}, 
(ii) follows from \cite{Burago}, and (iii) from \cite{Mu1,Mu2,MR}.

If Uhlmann's conjecture 
\cite{Uhlmann}, that the scattering relation determines the isometry type
of non-trapping compact manifolds with non-empty boundary, can be proven,
then Corollary \ref{cor: main_thm} holds for more general class of manifolds.

%In the two dimensional case the same is true even if the metric tensors $g$ and $g'$
%are not close to the Euclidean metric \cite{PU}. That is, we may take $\epsilon_2 = \infty$ in the above corollary.

The problem of determining the metric $g$ (possibly up to a diffeomorphism)
 with given the measurement $L h_0$ with only one function $h_0(t,x)$
is a formally determined inverse problem. Indeed, the formally computed
``dimension of the data'', that is the dimension of  $(T_0,T) \times \p M$, is $n$ and
coincides with dimension of the set $M$ on which the unknown functions $g_{jk}(x)$ are defined. 

The formally determined inverse problems have been studied in many cases. 
For instance, two dimensional Calderon's inverse problem \cite{AP,ALP,IUY,N1,S}
%, i.e. the inverse problem for the conductivity equation 
is formally determined.
The same is true for the related inverse problem for the Schr\"odinger equation 
in two dimensions \cite{Bu}.
The corresponding inverse problems in dimension $n \geq 3$, see \cite{C,KSU,LTU,N2,SU} and references in \cite{GKLUrev},
are over-determined, that is, the dimension of the data is larger than the dimension of the unknown object.
Similar classification holds for the elliptic inverse problems on Riemannian manifolds \cite{GT,GT2,LTU,LU,LeU}. 
Moreover, the inverse travel time problems, i.e. boundary rigidity problem,
see \cite{KLU, Mi, Mu1, Mu2, MR, Ro, lensJAMS}, is formally determined
in dimension $n=2$ and overdetermined for $n\geq 3$.
%We note that in dimension $n\geq 3$ most of the formally determined problems
%concern determination of inclusions.
%Similarly, in dimension $n\geq 3$ the inverse scattering problems with complete data for
%the Schr\"odinger equation are overdetermined but the backscattering problem
%is formally determined. There are also many formally determined
%inverse problems related to obstacle scattering or inclusion detection.

Inverse problems in time domain related to the Laplace-Beltrami operator
$\Delta_g$, namely the inverse boundary value problem 
for the wave, heat, and the dynamical Schr\"odinger equations with 
Dirichlet-to-Neumann as data, see \cite{AKKLT, BeK, KK, KKL}, are overdetermined in dimensions
$n\geq 2$. 
However, these problems are equivalent to the inverse
boundary spectral problem, see \cite{KKLM},
and assuming that the eigenvalues are simple, 
Dirichlet-to-Neumann map at a generic Dirichlet boundary value determines
the boundary spectral data \cite{La,La2,Ram}. 
Thus under generic conditions on the spectrum and on the boundary value 
(that is, under conditions that the these data belong in some open and dense set)
it is possible to solve a formally determined inverse problem in time domain. 

We point out that in this paper we do not impose any generic conditions on the geometry
and we give an explicit constuction of the boundary source. 
The  boundary source considered in this paper is based on the idea
of imitating a realization of 
white noise, and due to the many useful properties of white noise process, we hope 
that the constructed source  may be useful in the study of other 
inverse problems requiring generic assumptions on the source.

Another formally determined hyperbolic inverse problem, 
namely measuring Neumann data 
when the initial data $(u|_{t=0},\p_t u|_{t=0})$ is non-zero and
satisfies subharmonicity or positivity conditions, has been studied using Carleman estimates 
\cite{Bellassoued,Imanuvilov,Klibanov}.
The present paper is closely related
to these studies, but we emphasize that we assume that the initial data for $u$ vanishes.

\section{Pseudorandom noise as a source}
\label{sec:pseudo_white_noise}

In this section we define a special source $h_0(t, x)$ %, $x \in \p M, t = 0$, 
which we call the {\it pseudorandom noise}. 
%{\newtext Roughly speaking, the idea of this construction is to
%produce a source combining many Dirac delta distributions (which amplitudes we call the "peaks") in such a way that if the
%source is multiplied with some unknown, non-vanishing smooth function, then the location
%of each peak can be identified from the amplitudes of the peaks close to it.} 
The specific  assumptions on the amplitudes are explained in Section 
\ref{sec:scattering_relation}.
An important feature of the pseudorandom noise is that it is supported only on a single point in time.

\begin{definition}
\label{def:pseudorandom_noise}
Let $x_j\in \p M$, $j = 1, 2, \dots$, 
be a dense sequence of disjoint points in $\p M$,
and let $a_j \in \R$, $j = 1, 2, \dots$, $\sum_{j=1}^\infty |a_j|<\infty$ be a 
 sequence of disjoint numbers.

We define the pseudorandom noise on 
$(x_j)_{j=1}^\infty \subset \p M$ with coefficients $(a_j)_{j=1}^\infty\subset \R$ as the following generalized function
on $\R \times \R^n$: 
\begin{equation*}
h_0(t,x) := \sum_{j=1}^\infty a_j \delta(t) \delta_{x_j}(x), \quad (x,t) \in \R^{n+1},
\end{equation*}
where $\delta(t)$ and $\delta_{x_j}(x)$ are Dirac delta distributions on $\R$ and $\R^n$, respectively.
\end{definition}

It is rather straightforward to show that $h_0$ is well-defined.
First, it is well known that $\delta(t) \in H^{-1} (\R)$ and $\delta_{x_j}(x) \in C(\R^n)'$.
Moreover, $H^1_{p'} (\R^n) \subset C(\R^n)$ when $1 > n/p'$ due to \cite[Thm.\ 2.8.1]{TRI}.
According to \cite[Thm. 2.6.1]{TRI} the dual space satisfies $(H^1_{p'} (\R^n))' = H^{-1}_p(\R^n)$
with $1/p' = 1 - 1/p$ and hence $C(\R^n)'\subset H^{-1}_p(\R^n)$ for $1< p < \frac{n}{n-1}$.
Since $\sum_{j=1}^\infty |a_j|<\infty$ we have 
\begin{equation*}
	\sum_{j=1}^\infty a_j \delta_{x_j}(x) \in H^{-1}_p (\R^n).
\end{equation*}
This yields that for any $p \in (1, \frac{n}{n-1})$ and $\epsilon > 0$ the pseudorandom noise $h_0$ satisfies
\begin{equation}
	h_0 \in \widetilde H^{-1} (-\epsilon, \epsilon) \otimes \widetilde H^{-1}_p (M).
\end{equation}

The spatial structure of the pseudorandom noise can be motivated by the structure the white noise.
In the 1-dimensional radar imaging models, white noise signals are considered to be optimal sources when imaging a stationary scatterer
\cite{radarprinciples}.
This is due to the fact that different translations of the white noise signal are uncorrelated.
In a similar fashion we have the following property for the pseudorandom noise $h_0$:
for each $x_{j_0}$ and each sequence $(x_{j_k})_{k = 1}^\infty$ 
converging to $x_{j_0}$ and satisfying $x_{j_k} \ne x_{j_0}$ for all $k \in \Z_+$,
it holds that $a_{j_k}\to 0$.
This property will be crucial in what follows. 
 
Moreover, a natural strategy to choose the points $x_j$ is by random sampling.
The term pseudorandom refers to the fact that
the algorithmic generators of random numbers use, in fact, a deterministic function 
to produce a sequences of numbers through
so mixing process, that the user of the algorithm can consider the numbers to be analogous to independent 
samples of a random variable.  
In this manner, the pseudorandom noise can be seen as an imitation of a realization of a noise process. 

Another source of inspiration for us was a rather new measurement paradigm called compressed sensing \cite{CRT,donoho}, where one 
aims for a sparse reconstructions of a linear problem using a small number of noisy measurements. 
%There are rigorous studies \cite{CRT} discussing how the measurements should be chosen 
%in order to obtain a stable reconstruction. 
We point out that by using the pseudorandom noise one can compress
the measurements $L f_{x_j}$ with point sources $f_{x_j}$
into a single measurement $L h_0$.

\section{Measurement map}
\label{sec:measurement_map}
In this section we prove that the measurement is well-defined when
we have pseudorandom noise as source.

{Next, we consider the 
operator $W : f \mapsto u$ mapping $f$ to the solution of the equation (\ref{eq:wave_eq_1}).
We call such operator the solution operator for the equation (\ref{eq:wave_eq_1}).}
First, we note that by 
\cite[Thm.\ 23.2.2]{hormander:pde}, the 
operator $W : f \mapsto u$ mapping $f$ extends in a unique way to a  continuous linear operator 
\begin{equation}\label{eq: oli lemma 1}
W:L^1((T_0, T); H^s(\R^n)) \to C([T_0, T]; H^{s+1}(\R^n)), \quad s \in \R.
\end{equation}
Moreover, if $f \in C^\infty([T_0, T] \times \R^n)$
and $\supp(f) \subset\subset (T_0, T] \times \R^n$, that is,
$\supp(f)$ is a compact subset of  $(T_0, T] \times \R^n$,
 then 
$Wf \in C^\infty([T_0, T] \times \R^n)$.

We will compose the operator $W$ with the one-sided inverse $\I$ of the derivative
$\p_t$, which is given by
% considered in the next simple lemma. 
%\begin{lemma}
%\label{lem:antiderivativeI}
%The integral operator 
\begin{equation*}
\I u(t) := \int_{T_0}^t u(t') dt', \quad u \in C_c^\infty(T_0, T).
\end{equation*}
One sees easily that this operator has a unique continuous linear extension $\I:\widetilde H^{-1}(T_0, T) \to L^2(T_0, T)$.

Next we prove that the measurement map $L$ has unique continuous 
extension 
\begin{equation}
\label{eq:extending_measurement_map_spaces}
\widetilde H^{-1} (T_0,T) \otimes H^{-1}_p (\R^n) 
\to \D'((T_0,T) \times \p M).
%\to H^{-2m - 1} ((T_0,T); L^p(\p M)).
\end{equation}

\begin{proof}[Proof of Lemma \ref{lem:extending_measurement_map}]
For sufficiently large $ z \in \R_+$, operator $ z - \Delta_g$ is 
an isomorphism between spaces $H^{s+2}(\R^n)$ and $H^s(\R^n)$
as well as between spaces $H_p^{s+2}(\R^n)$ and $H_p^s(\R^n)$ for all integers $s$
by \cite{Sh1}.

By the definition of $L$, we have that $L = \Tr \circ W$, where 
$\Tr$ is the trace operator
\begin{equation*}
\Tr (u )= u|_{(T_0, T) \times \p M},
\quad u \in C^\infty((T_0, T) \times \R^n).
\end{equation*}

Let $f \in C_c^\infty( (T_0,T) \times \R^n)$.
Then the solution $u=Wf$ of the wave equation $(\p_t^2-\Delta_g)u=f$ can be written in the form
\begin{align} \label{eq:extension}
Wf &= ( z-\p_t^2)^m ( z - \Delta_g)^{-m} W f
+ \sum_{j=0}^{m-1} ( z-\p_t^2)^j ( z - \Delta_g)^{-1-j} f.
\end{align}
Now $f=\p_t\I f$, where $\I f $ is $C^\infty$-smooth and satisfies $\supp (\I f) \subset\subset (T_0, T] \times \R^n$.
By (\ref{eq: oli lemma 1}), $W \I f$ is $C^\infty$-smooth and 
$\p_t W \I f = W \p_t \I f = Wf$.
Hence
\begin{align} \label{eq:extension2}
Lf %&= \Tr W f  \\\nonumber
&= \p_t ( z-\p_t^2)^m \Tr ( z - \Delta_g)^{-m} W \mathcal \I f
+ \sum_{j=0}^{m-1} ( z-\p_t^2)^j \Tr ( z - \Delta_g)^{-1-j} f.
\end{align} 

Let us  next consider terms appearing in (\ref{eq:extension2}).
First we consider extension of  the operator 
\begin{align}
\label{eq:extension_terms_of_sum}
\sum_{k=1}^N \phi_k \otimes \psi_k 
&\mapsto \sum_{k=1}^N ( z-\p_t^2)^j \Tr ( z - \Delta_g)^{-1-j} (\phi_k \otimes \psi_k)
\\\nonumber&= \sum_{k=1}^N (( z-\p_t^2)^j \phi_k) \otimes (\Tr ( z - \Delta_g)^{-1-j} \psi_k)),
\quad j = 0, \dots, m-1,
\end{align}
mapping $C_c^\infty(T_0,T) \otimes C_c^\infty(\R^n)$ to $C^\infty ((T_0,T) \times \p M)$.
%has continuous extension in spaces 
%(\ref{eq:extending_measurement_map_spaces}).
By \cite[Thm 4.7.1]{TRI} the maps
\begin{equation*}
H^{-1}_p (\R^n) 
\xrightarrow{( z - \Delta_g)^{-1-j}} H^1_p (\R^n) 
\xrightarrow{\Tr} B_{p,p}^{1-1/p} (\p M)
\end{equation*}
are continuous, where $B_{p,p}^{1-1/p} (\p M)$ is the Besov space on $\p M$.
Thus
%As the embeddings 
%\begin{align*}
%&\widetilde H^{-1} (T_0,T) \hookrightarrow \D'(\R) \hookrightarrow \D'(T_0, T),
%\\& B_{p,p}^{1-1/p} (\p M) \hookrightarrow \D'(\p M),
%\\& \D'(T_0, T) \otimes \D'(\p M) \hookrightarrow \D'( (T_0, T) \times \p M)
%\end{align*}
%are continuous, 
% For the last one see e.g. [Treves, Thm. 51.7], a much stronger result. 
the operator (\ref{eq:extension_terms_of_sum}) has a continuous extension
in spaces (\ref{eq:extending_measurement_map_spaces}).
%As $p < 2$ we have continuous embeddings [Triebel Thm 4.6.1]
%\begin{equation*}
%B_{p,p}^{1-1/p} (\p M) 
%\subset H_p^{1-1/p} (\p M) 
%\subset L^p (\p M).
%\end{equation*}
%%\note{Actually the results in Triebel's book are only for domains not for manifolds.}
%
%Operator $( z-\p_t^2)^j$ is continuous $\tilde H^{-1} (T_0,T) \to H^{-2m + 1} (T_0,T)$
%when $j \le m-1$. 
%%\note{TODO Check continuity of $\otimes$.}
%Tensor product is continuous 
%\begin{align*}
%H^{-2m + 1} (T_0,T) \times L^p (\p M) \to  H^{-2m + 1} ( (T_0,T); L^p(\p M) ),
%\end{align*}
%and thus the mapping (\ref{eq:extension_terms_of_sum}) is continuous 
%\begin{align*}
%\tilde H^{-1} (T_0,T) \times H^{-1}_p (\R^n) \to H^{-2m + 1} ( (T_0,T); L^p(\p M) ).
%\end{align*}

Next, consider the extensions of  the operator 
\begin{align}
\label{eq:extension_leading_term}
\sum_{k=1}^N \phi_k \otimes \psi_k \mapsto 
&\sum_{k=1}^N \p_t ( z-\p_t^2)^m \Tr ( z - \Delta_g)^{-m} W \mathcal ( (\I\phi_k) \otimes \psi_k) 
%\\&\quad= \p_t ( z-\p_t^2)^m ( z - \Delta_g)^{-m} \Tr W ( (\mathcal I \phi) \otimes \psi) 
\end{align}
mapping $C_c^\infty(T_0,T) \otimes C_c^\infty(\R^n)$ to $C^\infty ((T_0,T) \times \p M)$.
%has continuous extension
%in spaces (\ref{eq:extending_measurement_map_spaces}).
As $-1 - n/p > -1 - n $ we have by \cite[Thm. 2.8.1]{TRI} 
a continuous embedding 
$
H^{-1}_p (\R^n) \hookrightarrow H^{-1-n/2} (\R^n).
$
Moreover, the operator $\mathcal I:\widetilde H^{-1} (T_0,T) \to L^2 (T_0,T)$ and
the embedding
$
L^2(T_0,T) \otimes H^{-1-n/2} (\R^n) \hookrightarrow L^2 ( (T_0,T); H^{-1-n/2} (\R^n))
$
are continuous. Thus,
by (\ref{eq: oli lemma 1}),
\begin{align*}
W \mathcal I:\widetilde H^{-1} (T_0,T) \otimes H^{-1}_p (\R^n) 
\to C ( [T_0,T];  H^{-n/2} (\R^n))
\end{align*}
is continuous.

Thus, as  $(1 - \Delta_g)^{-m}:C ( [T_0,T];  H^{-n/2} (\R^n)) 
\to C ( [T_0,T];  H^{-n/2 + 2m} (\R^n))$ is continuous and 
 $-n/2 + 2m > 1/2$, we see that the operator 
\begin{equation*}
\Tr (1 - \Delta_g)^{-m} W \mathcal I:\widetilde H^{-1} (T_0,T) \otimes H^{-1}_p (\R^n) 
\to C ( [T_0,T];  L^2(\p M))
\end{equation*}
is continuous.

Combining the above results, we see that the operator (\ref{eq:extension2})
has a continuous extension to
spaces (\ref{eq:extending_measurement_map_spaces}).
As the spaces $C_c^\infty(T_0,T)$ and $C_c^\infty(\R^n)$ are dense 
in $\widetilde H^{-1} (T_0,T)$ and $H^{-1}_p (\R^n)$, respectively,
we see that  the continuous extension of $L$ is unique. 
\end{proof}

\section{Inner product of a solution and a source}

%In this section we generalize the 
%Blagovescenskii identity to compute  
 %inner product between source and solution, see e.g.\ \cite{KKL}.
%Later it yields a crucial component of our analysis, namely, the connection of 

\begin{lemma}
\label{lem:integration_by_parts_smooth}
Let $f \in C_c^\infty((T_0, T) \times M)$, $t_0 \in (T_0, T)$
and let $w \in C^\infty([T_0, t_0] \times \R^n)$ satisfy
\begin{equation*}
(\p_t^2 - \Delta_g) w = 0, \quad \text{ in } (T_0, t_0) \times \R^n.
\end{equation*}
Then 
\begin{align*}
&\int_{T_0}^{t_0} \int_{\R^n} f(t,x) w(t,x) dt dV(x)
\\&\quad\quad= 
\int_{\R^n} \l( (\p_t Wf)(t_0, x) w(t_0, x) - (Wf)(t_0, x) (\p_t w)(t_0, x) \r) dV(x)
\end{align*}
where $dV(x) = |g|^{1/2} dx$ is the Riemannian volume measure of $(\R^n, g)$
and $W:f\mapsto u$ is the solution operator of wave  equation (\ref{eq:wave_eq_1}).
\end{lemma}
\begin{proof}
By finite speed of propagation of waves, see e.g. \cite[pp.\ 150-156]{Lady}, $\supp (Wf(t))$ is compact in $\R^n$. 
The claim follows by integration by parts
\begin{align*}
&\int_{\R^n} ( (\p_t u)(t_0,x)w(t_0,x)- u(t_0,x)(\p_t w)(t_0,x))dV(x)
\\&\quad -\int_{\R^n} ( (\p_t u)(T_0,x)w(T_0,x)- u(T_0,x)(\p_t w)(T_0,x))dV(x)
\\&=\int_{(T_0, t_0) \times \R^n}(
(\p_t^2-\Delta_g)u(t,x)\, w(t,x)-u(t,x)\, (\p_t^2-\Delta_g)w(t,x))dtdV(x)
\\&=\int_{(T_0, t_0) \times \R^n} f(t,x)w(t,x)dtdV(x).
\end{align*}
\end{proof}

Next, we will prove a generalization of the previous lemma 
for non-smooth sources $f$.
Denote by $B(0,R)=\{ x \in \R^n ; |x| < R\}$ the Euclidean ball. 
The finite speed of propagation for wave equation,
yields that there is $R > 0$ such that 
all $f \in C_c^\infty((T_0, T) \times M)$ 
satisfy $\supp(Wf) \subset \subset (T_0, T] \times B(0,R)$. 
We define 
\begin{equation}
\label{def:exterior_domain}
\Omega := B(0,R) \setminus \bar M.
\end{equation}

Below, we use the fact (see
\cite[Thm.\ 7.2.3/6, Thm.\ 5.6.3/6]{evans}) that
the  operator $W_\Omega : h \mapsto v$ mapping $h$ to the solution of the equation 
\begin{align} \label{eq:wave_eq_exterior_domain}
&(\p_t^2 - \Delta_g)v(t,x) = 0
    \quad {\rm in} \; (T_0, T) \times \Omega,
\\&v|_{(T_0, T) \times \p \Omega} = h, \nonumber
\\&v|_{t=T_0}=0,\quad \p_t v|_{t=T_0}=0. \nonumber
\end{align}
is continuous 
%\begin{equation*}
$W_\Omega:C_c^\infty((T_0, T) \times \p \Omega) \to C^\infty([T_0, T] \times \bar \Omega).$

We let $t_0 \in (T_0, T)$ and denote
\begin{equation}
\label{def:surface_Sigma}
\Sigma := \{t_0\} \times \Omega
\end{equation}
We denote the trace on $\Sigma$ by $\Tr_\Sigma$,
that is, we define $(\Tr_\Sigma u)(x) := u(t_0, x)$.
Let $\nu=\nu(z)$ denote the exterior unit normal vector of $\p M$ at $z$.

Moreover, let $U$ be an open subset (or a submanifold) of $\R^n$
and let us denote by $dV$ (or $dS$) the Riemannian volume
measure of $(U, g)$.
We embed the test functions into the spaces of distribution by using
the inner product of the space $L^2(U; dV)$, that is,
we identify $u\in C^\infty_0(U)$ with the distribution
\begin{equation}
\label{eq:embedding_test_functions}
\psi\mapsto \int_U u(x) \psi(x)\,dV(x).
\end{equation}
We will denote
the distribution pairing of $u\in \D'(U)$ and $\psi \in C^\infty_0(U)$
by $(u,\psi)_{\D'(U)}$ and use analogous notations for other distribution
pairings.

\begin{lemma}
\label{lem:extension_of_measurement}
Let $t_0 \in (T_0, T)$ and define $\Sigma$ by (\ref{def:surface_Sigma}).
Then operators $\Tr_\Sigma W_\Omega$ and $\Tr_\Sigma \p_t W_\Omega$ 
have unique continuous extensions 
%\begin{equation*}
$\E'((T_0,t_0) \times \p \Omega) \to \D'(\Omega).$
%\end{equation*}

\end{lemma}
\begin{proof} Let $v$ satisfy (\ref{eq:wave_eq_exterior_domain}).
Consider a function 
$w \in C^\infty([T_0, t_0] \times \bar \Omega)$
such that $(\p_t - \Delta_g) w = 0$ in $(T_0, t_0) \times \Omega$ 
and $w|_{(T_0, t_0) \times \p \Omega} = 0$.
Then
\begin{align*}
0 &= \int_{\Omega \times (T_0, t_0)} ((\p_t - \Delta_g) v) w - v ((\p_t - \Delta_g) w) dV(x) dt
\\&= \l[ \int_\Omega ((\p_t v) w - v (\p_t w)) dV(x) \r]_{t = T_0}^{t = t_0}
+ \int_{\p \Omega \times (T_0, t_0)}( (\p_\nu v) w - v (\p_\nu w)) dS(x) dt
\\&= \l. \int_\Omega ((\p_t v) w - v (\p_t w)) dV(x) \r|_{t = t_0}
- \int_{\p \Omega \times (T_0, t_0)} h (\p_\nu w) dS(x) dt,
\end{align*}
where
$\p_\nu$ is the normal derivative on $\p \Omega$,

Denote by $W_1 : f_1 \mapsto w$ 
the solution operator of the equation
\begin{align*}
&(\p_t - \Delta_g) w(t,x) = 0
    \quad {\rm in} \; (T_0, t_0) \times \Omega, 
\\&w|_{(T_0, t_0) \times \p \Omega} = 0, \nonumber
\\&w|_{t=t_0}=f_1,\quad \p_t w|_{t=t_0}=0. \nonumber
\end{align*}
Operator 
%\begin{equation*}
$W_1:C_c^\infty(\Omega) \to C^\infty([T_0, t_0] \times \bar \Omega),$
%\end{equation*}
is continuous,
as can be seen using \cite[Thm. 7.2.3/6, Thm. 5.6.3/6]{evans}.
Hence the operator 
\begin{equation*}
\p_\nu W_1:C_c^\infty(\Omega) \to C^\infty([T_0, t_0] \times \p \Omega),\quad
f\mapsto \p_\nu W_1 f|_{\p \Omega}
\end{equation*}
is continuous. Moreover, 
\begin{equation*}
(\Tr_\Sigma \p_t W_\Omega h, f_1)_{L^2(\Omega; dV)}
= (h, \p_\nu W_1 f_1)_{L^2((T_0, t_0) \times \p \Omega; dt \otimes dS)},
\end{equation*}
where $\p_\nu$ is the normal derivative on $\p \Omega$.
We define the extension of $\Tr_\Sigma \p_t W_\Omega$ by identifying
it with the transpose 
$(\p_\nu W_1)^t:\E'((T_0, t_0) \times \p \Omega))\to \D'(\Omega) $ 
of the operator $\p_\nu W_1:C_c^\infty(\Omega) \to C^\infty([T_0, t_0] \times \p \Omega)$.

Similarly, we define the extension of $\Tr_\Sigma W_\Omega$ by the transpose
$(\p_\nu W_2)^t:\E'((T_0, t_0) \times \p \Omega))\to \D'(\Omega) $
of $\p_\nu W_2:C_c^\infty(\Omega) \to C^\infty([T_0, t_0] \times \p \Omega)$, where 
$W_2 : f_2 \mapsto w$ is the solution operator of the equation
\begin{align*}
&(\p_t - \Delta_g) w(t,x) = 0
    \quad {\rm in} \; (T_0, t_0) \times \Omega, 
\\&w|_{(T_0, t_0) \times \p \Omega} = 0, \nonumber
\\&w|_{t=t_0}=0,\quad \p_t w|_{t=t_0}=-f_2. \nonumber
\end{align*}
\end{proof}

Denote by $d_{\Omega}(x,y)$, $x, y \in \overline \Omega$, the distance function of Riemannian manifold 
$(\overline \Omega, g|_{\overline \Omega})$. Next we generalize the result of Lemma \ref{lem:integration_by_parts_smooth}
for a larger class of functions.

\begin{lemma}
\label{lem:integration_by_parts}
Let $t_0 \in (0, T)$ and  $\epsilon > 0$ satisfy
$[-\epsilon, \epsilon] \subset (T_0, t_0)$.
Define $\Sigma$ by (\ref{def:surface_Sigma}).
Let $f \in \widetilde H^{-1} (-\epsilon,\epsilon) \otimes \widetilde H^{-1}_p (M)$
and  $w \in C^\infty([T_0, t_0] \times \R^n)$ satisfy
\begin{equation*}
(\p_t^2 - \Delta_g) w = 0, \quad \text{ in } (T_0, t_0) \times \R^n.
\end{equation*}
Suppose that $w(t_0), \p_t w(t_0) \in C_c^\infty(\Omega)$,
and let $\chi \in C_c^\infty(T_0, t_0)$ satisfy 
$\chi = 1$ in a neighborhood of $[-\epsilon, t_0 - r]$, where
\begin{equation*}
r := d_\Omega\big (\supp(w(t_0)) \cup \supp(\p_t w(t_0)), \p \Omega\big).
\end{equation*}
Then 
\begin{align}
\label{eq:integration_by_parts}
&(%|g|^{1/2} 
f, w)_{\E'(\R^n \times (T_0, t_0))} 
\\&\quad= (\Tr_\Sigma \p_t W_\Omega \chi L f, w)_{\D'(\Omega)} 
- (\Tr_\Sigma W_\Omega \chi L f, \p_t w)_{\D'(\Omega)},
\nonumber
\end{align}
where we have defined $L f = 0$ on $\p B(0,R)$.
Here we regard $\Omega$ as Riemannian manifold $(\Omega, g|_{\Omega})$.
\end{lemma}
\begin{proof}
We suppose first that $f \in C_c^\infty((-\epsilon, \epsilon) \times M)$. 
Recall that $W$ is solution operator of wave  equation (\ref{eq:wave_eq_1}).
Then
$W f(\cdot, t) = 0$ if $t < -\epsilon$, and
\begin{equation*}
L f = \Tr_{\p \Omega} W f = \chi \Tr_{\p \Omega} W f, \quad \text{in $(T_0, t_0 - r) \times \p \Omega$},
\end{equation*}
where $\Tr_{\p \Omega}$ is the trace on $(T_0, T) \times \p \Omega$.
As $\Omega \cap \overline M = \emptyset$, we have that 
$(\p_t^2 - \Delta_g) W f = 0$ in $(T_0, T) \times \Omega$.
By uniqueness of the solution of (\ref{eq:wave_eq_exterior_domain})
\begin{equation*}
W_\Omega \chi \Tr_{\p \Omega} W f = W f, \quad \text{in $(T_0, t_0 - r) \times \Omega$}.
\end{equation*}
By finite speed of propagation 
\begin{align*}
\Tr_\Sigma \p_t^j W_\Omega \chi \Tr_{\p \Omega} W f = \Tr_\Sigma \p_t^j W f,
\quad j = 0, 1,
\end{align*}
on $\{t_0\} \times \supp(w(t_0)) \cup \supp(\p_t w(t_0))$.
By Lemma \ref{lem:integration_by_parts_smooth}, (\ref{eq:integration_by_parts})
holds.

Then the claim follows as the  embeddings 
\begin{equation*}
C_c^\infty(-\epsilon, \epsilon) \hookrightarrow \widetilde H^{-1} (-\epsilon,\epsilon),
\quad C_c^\infty(M) \hookrightarrow \widetilde H^{-1}_p (M)
\end{equation*}
are dense and operators 
$(\Tr_\Sigma \p_t^j W_\Omega) \chi  L
: \widetilde H^{-1} (-\epsilon,\epsilon) \otimes \widetilde H^{-1}_p (M)
\to \D'((T_0, t_0) \times \p \Omega)$, $j = 0, 1$, are
%and $(\Tr|_{t=t_0} W_\Omega) \circ \chi \circ L$ are 
continuous.
\end{proof}

\section{Gaussian beams}

We consider solutions of wave equation 
which are known as Gaussian beams \cite{Babich1,Babich2,Ral}.
These solutions have been constructed
to analyze the propagation of singularities for the wave equation in the presence of caustics.
Here we use
Gaussian beams as an auxiliary technical tool to analyze singularities
in the measurements.  
%Our presentation follows that in [KKL].

\begin{definition}
\label{def:formal_gaussian_beam}
%Let $(y,\eta)\in T\R^n$ satisfy $|\eta|_g=1$, and 
Let $\epsilon > 0$, $N \in \N$ 
and let $\gamma$ be a unit speed geodesic on $(\R^n,g)$.
%with $\gamma(0)=y$, $\dot \gamma(0)=\eta$.
A formal Gaussian beam of order $N$ propagating along
geodesic $\gamma$ is a function $U^N_{\epsilon}$
of form
\begin{equation*}
U^N_{\epsilon}(t,x) = \epsilon^{-n/4}
\exp {\{ -(i\epsilon)^{-1}\theta(t,x) \} }
\sum _{m=0}^N u_m(t,x) (i\epsilon )^m,\quad t\in \R, x\in \R^n
\end{equation*}
satisfying the following properties:
The phase function $\theta$ and the
amplitude functions $u_m$, $m=0, 1, \dots , N$, are complex valued smooth functions. %of the variables $t$ and $x$.
The phase function $\theta$
satisfies the conditions
\begin{align*}
\theta (t,\gamma(t)) =0,
\quad \Im \theta(t,x) \geq C_0(t)d(x,\gamma(t))^2
\end{align*}
%where $d(\cdotp,\cdotp)$ is the distance function in the metric $g$ 
where $C_0(t)$ is a continuous strictly positive function. 
The amplitude function $u_0$ satisfies
$
u_0(t,\gamma(t)) \ne 0.
$
Finally, for any compact set 
$K \subset\subset \R \times \R^n$ there is a constant $C>0$ such that the inequality
\begin{equation*}
|(\partial _t^2 - \Delta _g)U^N_{\epsilon}(t,x)|\leq C \epsilon ^{N-n/4}
\end{equation*}
is satisfied uniformly for $(t,x) \in K$.
\end{definition}

The construction of a formal Gaussian beam $U^N_{\epsilon}(t, x)$
is considered in detail e.g.\ in \cite[Sect.\ 2.4]{KKL}. 
Next, we recall the construction and pay attention to the properties
of Gaussian beams which we need later.

Let us write the geodesic $\gamma$ in the usual coordinates of $\R^n$ as 
$\gamma(t) = (\gamma^1(t), \dots, \gamma^n(t))$.
We construct the phase function $\theta(t,x)$ at each time $t \in \R$ 
in terms of a finite Taylor  expansion in the $x$ variable centered at $\gamma(t)$,
\begin{align*}
\theta(t,x) 
= \sum_{|\alpha| \le N} \frac{\theta_\alpha(t)}{\alpha!} (x-\gamma(t))^\alpha,
\end{align*}
where $\theta_\alpha$ are smooth functions and $N \in \N$.

Let $e_j=(\delta_{1j}, \dots, \delta_{nj})$ be multi-indexes with the value 1 at the $j$th place.
For clarity, we use the notation $p_j (t)= \theta_{e_j}(t)$ for the first order coefficients 
and the notation $H_{jk}(t) = \theta_\alpha(t)$, $\alpha={e_j + e_k}$, 
for the second order coefficients in the expansion of $\theta$. 

The construction of a formal Gaussian beam consists of the following steps. 
\begin{itemize}
\item[1.] We define  $\theta_0(t) = 0$ and 
$
p_j(t)= \sum_{k=1}^n g_{jk}(\gamma(t)) \dot \gamma^k(t),
$
that is, the first order coefficients $p_j(t)$ are
the covariant representation of the velocity vector $\dot \gamma$.

\item[2.] The symmetric matrix $H(t)=[H_{jk}(t)]_{j,k=1}^n$ of the second order coefficients 
are obtained by solving a Riccati equation, or an equivalent system
of ordinary differential equation. We write
%\begin{equation*}
$H (t) = Z(t) Y(t)^{-1},$
%\end{equation*}
where the pair of complex $n\times n$ matrices $(Z(t), Y(t))$ is the solution of the system of ordinary differential
equations,
\begin{align*}
&\frac d{dt} Y(t) = B(t)^*  Y(t) + C(t) Z(t), \quad \left. Y \right|
_{t=0} = Y^0,\label {2.140}\\ \nonumber
&\frac d{dt} Z(t) = -D(t)Y(t) - B(t)Z(t), \quad
\left. Z \right| _{t=0} =Z^0.
\end{align*}
Here we choose the initial values to be $Z^0=iI$ and $Y^0=I$, where $I$ is the identity matrix
and $i$ is the imaginary unit.
%Note that we could choose the initial value $Z^0$ to be any symmetric complex matrix, $Z^0_{jk}=Z^0_{kj}$ having strictly positive  imaginary part. 
The matrices $B(t)$, $C(t)$, and $D(t)$ in $\R^{n \times n}$ 
have components given by the second derivatives of the Hamiltonian 
$h(x,p)=(\sum_{j,k=1}^n g_{jk}(x)p^jp^k)^{1/2}$ 
evaluated in the point $(x,p) = (\gamma(t), p(t))$
\begin{equation*}
B^j_l = \frac {\partial ^2h}{\partial x^l \partial p_j}; \quad
C^{jl} = \frac {\partial ^2h}{\partial p_j \partial p_l}; \quad
D_{jl} = \frac {\partial ^2h}{\partial x^j \partial x^l}.
\end{equation*}
The fact that the complex matrix $Y(t)$ is invertible for all $t\in \R$
is crucial for the construction, and is discussed in detail in \cite[Section 2.4]{KKL}.

\item[3.] The coefficients $\theta_\alpha(t)$ of order $|\alpha| = m \ge 3$ are
solved inductively, with respect to $m$. The coefficients  $ \theta_\alpha(t)$
are constructed using the coefficients $\tilde \theta_\alpha(t)$ defined so that
\begin{equation*}
\sum_{|\alpha|=m}\tilde \theta_\alpha(t)\tilde y^\alpha
=\sum_{|\alpha|=m}\theta_\alpha(t) (x-\gamma(t))^\alpha, 
\end{equation*}
for all $\tilde y=Y^{-1}(t)(x-\gamma(t))$, $y\in \C^n$.
We obtain the coefficients $\tilde \theta_\alpha(t)$ by solving 
successive linear systems of ordinary differential equations 
\begin{equation*}
\frac{d}{dt} \tilde\theta_\alpha(t)=K_\alpha(t),\quad \tilde\theta_\alpha(0)=0
\end{equation*}
where $K_\alpha(t),$ depend on $ \theta_\beta(t)$ with $|\beta|\leq m-1$, the matrix $H(t)$, vector $p(t),  $
and the metric $g_{jk}$ and its derivatives at $\gamma(t)$.

 \item[4.] When the phase function $\theta(t,x)$ is constructed, the
amplitude functions $u_n(t,x)$ are solved using 
 the transport equations, or equivalently, the following ordinary differential equations. Let 
\begin{equation*}
u_m (t,x)=\sum _{|\alpha|\leq N}\tilde u_{m,\alpha}(t)\tilde y^\alpha,\quad \tilde y=Y^{-1}(t)(x-\gamma(t))
\end{equation*}
where the coefficients $\tilde u_{m,\alpha}(t)$ are obtained by solving the successive  equations 
\begin{equation*}
\frac d{dt} \tilde u_{m,\alpha}(t) + r(t) \tilde u_{m,\alpha}(t) =
{\cal F}_{m,\alpha}(t),\quad \tilde u_{m,\alpha}(0)=\delta_{m,0}\delta_{|\alpha|,0},
\end{equation*}
where $r(t)$ and
${\cal F}_{m,\alpha}(t)$ depend on $ \tilde u_{m',\beta}$ with $|\beta |\leq |\alpha|+2$ and $m'\leq m-1$, the function $\theta(t,x)$,
the metric $g_{jk}$ and their derivatives at $(t,x)$, $x=\gamma(t)$.
 \end{itemize}

By the above construction, we have the following remark.

\begin{remark}
\label{rem:gaussian_beam_locality}
The phase function $\theta(t,x)$ and the amplitude functions $u_m(t,x)$
at time $t=0$ have the form
\begin{align*}
&\theta(0,x) = \sum_{j,k=1}^n g_{jk}(y) \eta^k (x^j-y^j) + i|x-y|^2,
\end{align*}
where $(y,\eta) = (\gamma(0), \dot \gamma(0))$ 
is the initial data of the geodesic $\gamma$,
$u_0(0,x) = 1,$ and $u_m(0,x) = 0$ for $m > 0$.
Hence $U^N_{\epsilon}(0,x)$ is dependent on the metric $g_{jk}$
only via $g_{jk}(y)$.
Moreover, $\p_t U^N_{\epsilon}(0,x)$, although of more complex form,
is dependent on the metric $g_{jk}$ only via $\p^\alpha g_{jk}(y)$ 
for a certain finite collection of multi-indices $\alpha \in \N^n$.
\end{remark}

If the coefficients of an ordinary differential equation
depend smoothly on some parameter so does the solution \cite{Amann}, and thus
we see using an induction that the phase function $\theta$ and the amplitude functions 
$u_m$ depend smoothly on the initial data $(y, \eta) = (\gamma(0), \dot \gamma(0))$ 
of the geodesic $\gamma$.
In particular, the amplitude function $u_0(t, x; y, \eta)$ satisfies
\begin{equation}
\label{rem:gaussian_beam_continuity_wrt_geodesic}
u_0\in C^\infty(\R \times \R^n \times S \R^n).
\end{equation}

To this far we have considered a formal Gaussian beam. By 
using continuous  dependency of the solution of the wave equation on the
source term, on obtains the following results, see e.g. \cite{KKL}:

Let $\gamma$ be a unit speed geodesic, $N \in \N$, $\epsilon > 0$
and let $U^N_{\epsilon}$ be a formal Gaussian beam of order 
$N$ propagating along geodesic $\gamma$.
Let $\chi \in C^\infty_0(\R^n)$
be a function which is identically one in a neighborhood
of $\gamma(0)$ and let $t_0 > 0$ and let $R$ be the radius in the equation (\ref{def:exterior_domain}).
Then for $j \in \N$ and $\alpha \in \N^n$ 
satisfying $j+|\alpha| < N-n/4$
there is $C>0$ such that the solution 
$w_{\epsilon}$ of the wave equation
\begin{align}
\label{eq:gaussian_beam_solution}
&(\partial_t^2 - \Delta_g) w_\epsilon(t,x) = 0,
\quad (t,x) \in (T_0, t_0) \times \R^n,
\\&w_\epsilon(t_0,x) = \chi(x) U^N_\epsilon(0,x),
\nonumber \quad \p_t w_\epsilon(t_0,x) = - \chi(x) \p_t U^N_\epsilon(0,x).
\end{align}
satisfies 
\begin{equation}
\label{eq:gaussian_beam_is_close_to_formal}
\sup_{x \in B(0,R), t \in (T_0, t_0)} 
|\partial_t^j \partial_x^\alpha(w_\epsilon(t_0 - t, x) 
- U_\epsilon^N(t, x))| 
\leq C \epsilon^{N-(j+|\alpha|)-n/4}.
\end{equation}
We call $w_\epsilon$ a Gaussian beam of order $N$ 
propagating along geodesic $\gamma$ 
backwards on time interval $(T_0, t_0)$.

\section{Determination of the travel times}
\label{sec:travel_times}

%Our next aim is to use the measurement $Lh_0$ to determine
%whether a geodesic $\gamma=\gamma_{y,\eta}$ sent from the point $y\in \Omega$ to the
% direction $\eta$ goes through the set $M$ and then 
% hits to some of the points $x_j\in \p M$. For this end, we will use the Gaussian beams.

\begin{lemma}
\label{lem:testing_with_gaussian_beam}
Let $w_\epsilon$ be a Gaussian beam of order $N \ge 1 + n/4$ 
propagating along geodesic $\gamma$ 
backwards on time interval $(T_0, t_0)$,
that is, let $w_\epsilon$ be the solution of (\ref{eq:gaussian_beam_solution}).
Let $h_0$ be the pseudorandom noise 
\begin{equation}\label{A formula}
h_0(t,x) = \sum_{j=1}^\infty a_j \delta(t) \delta_{x_j}(x).
\end{equation}
If $\gamma(t_0) \ne x_j$ for all $j = 1, 2, \dots$
then
\begin{equation*}
\lim_{\epsilon \to 0} 
\epsilon^{n/4} (%|g|^{1/2} 
h_0, w_\epsilon)_{\E'(\R^n \times (T_0, t_0))} = 0.
\end{equation*}
Moreover, if $\gamma(t_0) = x_j$ then
\begin{equation*}
\lim_{\epsilon \to 0} 
\epsilon^{n/4} (%|g|^{1/2} 
h_0, w_\epsilon)_{\E'(\R^n \times (T_0, t_0))} =
a_j u_0(t_0, x_j) 
|g|^{1/2}(x_j), 
\end{equation*}
where $u_0(t, x)$ % = u_0(t, x; \gamma(0), \dot \gamma(0))$
is the first amplitude function of a formal Gaussian beam 
propagating along geodesic $\gamma$.
\end{lemma}
We remind the reader that the test functions are embedded in $\E'(\R^n \times (T_0, T))$
using (\ref{eq:embedding_test_functions}).
\begin{proof}
By equation (\ref{eq:gaussian_beam_is_close_to_formal})
we have that
\begin{align*}
&\epsilon^{n/4} (%|g|^{1/2} 
h_0, w_\epsilon)_{\E'(\R^n \times (T_0, t_0))}  
\\\quad&= \epsilon^{n/4} \sum_{j=1}^\infty a_j U_\epsilon^N(t_0, x_j) 
|g|^{1/2}(x_j) + O(\epsilon)
\\\quad&= \sum_{j=1}^\infty a_j u_0(t_0, x_j) 
\exp {\{ -(i\epsilon)^{-1}\theta(t_0,x_j) \} }
|g|^{1/2}(x_j) + O(\epsilon).
\end{align*}
As $\Im \theta(t_0, x_j) \ge C_0(t_0) d(x_j, \gamma(t_0))$ we have that
\begin{align*}
|\exp {\{ -(i\epsilon)^{-1}\theta(t_0,x_j) \} }| = O(\epsilon),\quad\hbox{if }
\quad \gamma(t_0) \ne x_j.
\end{align*}

Suppose that $\gamma(t_0) = x_j$. Then 
$\exp {\{ -(i\epsilon)^{-1}\theta(t_0,x_j) \} } = 1$
and there is a constant $C > 0$ depending on $\gamma$ and $t_0$ such that
\begin{align*}
&|\epsilon^{n/4} (%|g|^{1/2} 
h_0, w_\epsilon)_{\E'(\R^n \times (T_0, t_0))}
- a_j u_0(t_0, x_j) 
%\exp {\{ -(i\epsilon)^{-1}\theta(t_0,x_j) \} }
|g|^{1/2}(x_j)|
\\&\le 
C \sum_{k=1}^{j-1} |a_k| |\exp {\{ -(i\epsilon)^{-1}\theta(t_0,x_k) \} }|
+ C\sum_{k=j+1}^{l} |a_k| |\exp {\{ -(i\epsilon)^{-1}\theta(t_0,x_k) \} }|
\\&\quad + C\sum_{l+1}^\infty |a_l| + O(\epsilon).
\end{align*}
We may first choose large $l \in \N$ and then small $\epsilon > 0$
so that the above three sums are arbitrary small.
The case, $\gamma(t_0) \ne x_j$ for all $j = 1, 2, \dots$, is
similar.
\end{proof}

%\begin{figure}
%\begin{center}
%\includegraphics[width=0.9\textwidth]{distance}
%\end{center}
%\caption{Trajectory of a Gaussian beam 
%propagating along geodesic $\gamma(t) := \gamma(t; y_0, \eta_0)$
%backwards on time interval $(T_0, t_0)$.
%Here $S(y_0, \eta_0, t_0) \ne 0$
%and $\tau := \tau(y_0, \eta_0) < t_0$.
%}
%\end{figure}

Next we define an auxiliary function 
$S(y_0, \eta_0, t_0)$ which is non-zero if and only if  there is $j\in \Z_+$ such that
$\gamma_{y_0, \eta_0}(t_0)=x_j$. 

\begin{definition}
\label{def:hit_functions}
Let $(y_0,\eta_0)\in T\R^n$ be such that $y_0\in \Omega^{int}$ and $\|\eta_0\|_g=1$. 
We denote by $\gamma(t; y_0, \eta_0)=\gamma_{y_0,\eta_0}(t)$  the geodesic on $(\R^n,g)$ with
$\gamma(0)=y_0$, $\dot \gamma(0)=\eta_0$.
Moreover, let $w_\epsilon$ be a Gaussian beam of order $N \ge 1 + n/4$ 
propagating along $\gamma(t; y, \eta)$
backwards on time interval $(T_0, t_0)$.
We define 
\begin{align*}
S(y_0, \eta_0, t_0) &:= \lim_{\epsilon \to 0} 
\epsilon^{n/4} (%|g|^{1/2}
h_0, w_\epsilon)_{\E'(\R^n \times (T_0, t_0))},
%\\ \T(y_0, \eta_0) &:= \{ t \in (0, \infty] ; \gamma(t) \in \p M \}.
%\\ \tau(y_0, \eta_0) &:= \inf\{ t \in (0, \infty] ; \gamma(t; y_0, \eta_0) \in \p M\}.
\end{align*}
%Finally, let $J\subset T\R^n$ be set of those 
%$(y_0,\eta_0)\in T\R^n$ such that $y_0\in \Omega^{int}$, $\|\eta_0\|_g=1$, 
%$ \tau(y_0, \eta_0)<\infty$ such that $(\xi,\nu(z))_g<0$, where 
%\ba
%\xi=\dot \gamma_{y_0, \eta_0}(s_1),\quad z= \gamma_{y_0, \eta_0}(s_1),\quad s_1=\tau(y_0, \eta_0),
%\ea
%that is, $(y_0,\eta_0)\in J$ if the geodesic $ \gamma_{y_0, \eta_0}$ enters to $M$
%intersecting $\p M$ transversaly.
\end{definition}

\begin{lemma}
\label{lem:determination_of_S}
Let $(y_0, \eta) \in S \Omega$ and $t_0 \in (0,T)$.
Then $L h_0$ for pseudorandom noise $h_0$ 
and $(\Omega, g|_\Omega)$, given as a Riemannian manifold 
determine $S(y_0, \eta_0, t_0)$.
\end{lemma}
\begin{proof}
Let $w_\epsilon$ be a Gaussian beam of order $N \ge 1 + n/4$ 
propagating along the geodesic $\gamma(\cdot; y_0, \eta_0)$ 
backwards on time interval $(T_0, t_0)$.
We may choose the cut-off function $\chi$ 
in the equation (\ref{eq:gaussian_beam_solution})
so that $w_\epsilon(t_0), \p_t w_\epsilon(t_0) \in C_c^\infty(\Omega)$.
As $g|_\Omega$ is known, we have by Remark \ref{rem:gaussian_beam_locality} that
the initial data $w_\epsilon(t_0)$, $\p_t w_\epsilon(t_0)$ are known.
Moreover, operators $\Tr_\Sigma \p_t^j W_\Omega$, $j = 0,1$, $\Sigma := \{t_0\} \times \Omega$,
are known. After choosing a suitable cut-off function $\chi$ 
in Lemma \ref{lem:integration_by_parts}
we have that the measurement $Lh_0$ determines 
the distributional pairing $(%|g|^{1/2} 
h_0, w_\epsilon)_{\E'(\R^n \times (T_0, t))}$.
Hence $S(y_0, \eta_0, t_0)$ is determined.
\end{proof}

The implicit function theorem yields the following remark. 
Note that $t_0 \in \R$ in the remark is not necessarily the first 
intersection time. 

\begin{remark}
\label{rem:transverse_intersection}
Let $(y_0, \eta_0) \in S \R^n$ and $t_0 \in \R$ satisfy
\begin{equation*}
(\gamma(t_0; y_0, \eta_0), \dot \gamma(t_0; y_0, \eta_0)) \in \p_{\pm} S M.
\end{equation*}
Then there are neighborhoods $I \subset \R$ and $U \subset S \R^n$ of
$t_0$ and $(y_0, \eta_0)$ and a smooth map $\ell : U \to I$ such that for 
$t \in I$ and $(y, \eta) \in U$
\begin{align*}
\gamma(t; y, \eta) \in 
\begin{cases}
M, & \text{for}\ \pm t < \pm \ell(y,\eta),
\\ \p M, & \text{for}\ t = \ell(y,\eta),
\\ \Omega, & \text{for}\ \pm t > \pm \ell(y,\eta).
\end{cases}
\end{align*}
\end{remark}
%\begin{proof}
%In local coordinates taking $\gamma(t_0; y_0, \eta_0)$ to origin,
%\begin{align*}
%M &= \{x = (x^1, \dots, x^n);\ x^1 < 0\},
%\\ \p M &= \{x = (x^1, \dots, x^n);\ x^1 = 0\},
%\\ \Omega &= \{x = (x^1, \dots, x^n);\ x^1 > 0\},
%\\ \p_\pm S M &= \{ (x,\xi) = (x^1, \dots, x^n, \xi^1, \dots, \xi^n);\ x^1 = 0,\ \pm \xi^1 > 0 \}.
%\end{align*}
%As $\pm \p_t \gamma^1(t_0; y_0, \eta_0) > 0$ the implicit function theorem gives 
%neighborhoods $I \subset \R$ and $U \subset \R^n \times \R^n$ 
%and a smooth map $\ell : U \to I$ such that
%for $t \in I$ and $(y, \eta) \in U$
%\begin{equation*}
%\gamma^1(t; y, \eta) = 0 \quad \text{if and only if $t = \ell(y, \eta)$}. 
%\end{equation*}
%Moreover, by smoothness of the geodesic flow, the neighborhoods
%$I$ and $U$ can be chosen so that $\pm \p_t \gamma > 0$
%on $I \times U$.
%\end{proof}

We remind the reader that $\tau(x, \xi)$, $(x, \xi) \in T\R^n$, is defined as the first intersection time 
with $\p M$, that is 
\begin{equation*}
\tau(y_0, \eta_0) := \inf\{ t \in (0, \infty] ; \gamma(t; y_0, \eta_0) \in \p M \}.
\end{equation*} 
In the following, we use the Sasaki metric on the tangent bundle $TM$.

\begin{lemma}
\label{lem:semicontinuity_of_tau}
The first intersection times $\tau : S \Omega \to (0, \infty]$ 
and $\tau : \p_- SM \to (0, \infty]$ are lower semi-continuous.
\end{lemma}
\begin{proof}
Let us consider $\tau$ on $S \Omega$.
Let a sequence $( (y_j, \eta_j))_{j=1}^\infty \subset S \Omega$ converge to $(y_0, \eta_0) \in S \Omega$
as $j \to \infty$.
We denote $\gamma_j(t) := \gamma(t; y_j, \eta_j)$ and $\tau_j := \tau(y_j, \eta_j)$.

We will show next that 
$\liminf_{j \to \infty} \tau_j \notin (0, \tau_0)$.
Let $t \in (0, \tau_0)$. Then $\gamma_0(t) \notin \p M$ and
\begin{equation*}
d(\gamma_0(t), \p M) > 0.
\end{equation*}
Let $j \in \Z_+$. Suppose for a moment that $\tau_j < \infty$. 
Noting that $\gamma_j$ is unit speed and $\gamma_j(\tau_j) \in \p M$, we have
\begin{equation*}
|t - \tau_j| \ge d(\gamma_j(t), \gamma_j(\tau_j))
\ge d(\gamma_j(t), \p M). 
\end{equation*}
If $\tau_j = \infty$, then $|t - \tau_j| = \infty > d(\gamma_j(t), \p M)$.

The convergence $\gamma_j(t) \to \gamma_0(t)$, as $j \to \infty$, implies
that for large $j$ 
\begin{equation*}
|t - \tau_j| \ge d(\gamma_0(t), \p M)/2 > 0.
\end{equation*}
Hence $\liminf_{j \to \infty} \tau_j \ne t$ 
for all $t \in (0, \tau_0)$.

Clearly $\liminf_{j \to \infty} \tau_j \ge 0$,
and  there is $J \in \Z_+$ such that 
\begin{equation*}
\tau_j \ge d(y_j, \p M) \ge d(y_0, \p M)/2 > 0, \quad j \ge J.
\end{equation*}
Hence $\liminf_{j \to \infty} \tau_j \ne 0$ and 
$\liminf_{j \to \infty} \tau_j \ge \tau_0$.

Let us consider $\tau$ on $\p_- SM$.
Let a sequence $( (y_j, \eta_j))_{j=1}^\infty \subset \p_- SM$ converge to $(y_0, \eta_0) \in \p_- SM$
as $j \to \infty$.
We denote $\gamma_j(t) := \gamma(t; y_j, \eta_j)$ and $\tau_j := \tau(y_j, \eta_j)$.

Repeating the above argument, we see that $\liminf_{j \to \infty} \tau_j \notin (0, \tau_0)$.
Thus it is enough to show that $\liminf_{j \to \infty} \tau_j \ne 0$.

Remark \ref{rem:transverse_intersection}
gives neighborhoods $I \subset \R$ and $U \subset S \R^n$ of zero and $(y_0, \eta_0)$ 
and a map $\ell : U \to I$ of boundary intersection times.
%such that for $t \in I$ and $(y,\eta) \in U$
%\begin{equation*}
%\gamma(t; y, \eta) \in \p M \quad \text{if and only if $t = \ell(y, \eta)$}
%\end{equation*}
We denote $V := U \cap \p_- SM$. As
$\gamma(0; x, \xi) \in \p M$ for $(x, \xi) \in V$,
we have $\ell = 0$ in $V$. 
In particular $r := d(\ell(V), \R \setminus I) > 0$.
For large $j$, $(\gamma_j(0), \dot \gamma_j(0)) \in V$ and thus
\begin{equation*}
\gamma_j(t) \in M, \quad t \in (0, r).
\end{equation*}
Hence $\tau_j \ge r > 0$ for large $j$, and 
$\liminf_{j \to \infty} \tau_j \ge \tau_0$.
\end{proof}

We easily see the following continuity result for $\tau$.

\begin{lemma}
\label{lem:S_is_enough}
Let $( (y_j, \eta_j))_{j=1}^\infty \subset S \Omega$ converge to $(x, \xi) \in \p_- SM$ in the Sasaki metric.
Then $\lim_{j \to \infty} \tau(y_j, \eta_j) = 0$.
\end{lemma}
%\begin{proof}
%As $(x, \xi) \in \p_- SM$, Remark \ref{rem:transverse_intersection} gives
%neighborhoods $I \subset \R$ and $U \subset S \R^n$ of zero and $(x,\xi)$ and a map $\ell : U \to I$ 
%of boundary intersection times. 
%There is $J \in \Z_+$ such that for $j \ge J$ 
%\begin{equation*}
%(\gamma(0; y_j, \eta_j), \dot \gamma(0; y_j, \eta_j)) = (y_j, \eta_j) \in S \Omega \cap U.
%\end{equation*}
%Hence $0 < \ell(y_j, \eta_j)$ for $j \ge J$.
%Moreover, $\gamma(t; y_j, \eta_j) \in \Omega$ for $0 \le t < \ell(y_j, \eta_j)$ and $j \ge J$.
%Thus
%\begin{equation*}
%\tau(y_j, \eta_j) = \ell(y_j, \eta_j) \to \ell(x, \xi) = 0, \quad \text{as $j \to \infty$}.
%\end{equation*}
%\end{proof}

\begin{figure}
\begin{minipage}{0.5\linewidth}
\def\svgwidth{0.8\textwidth}
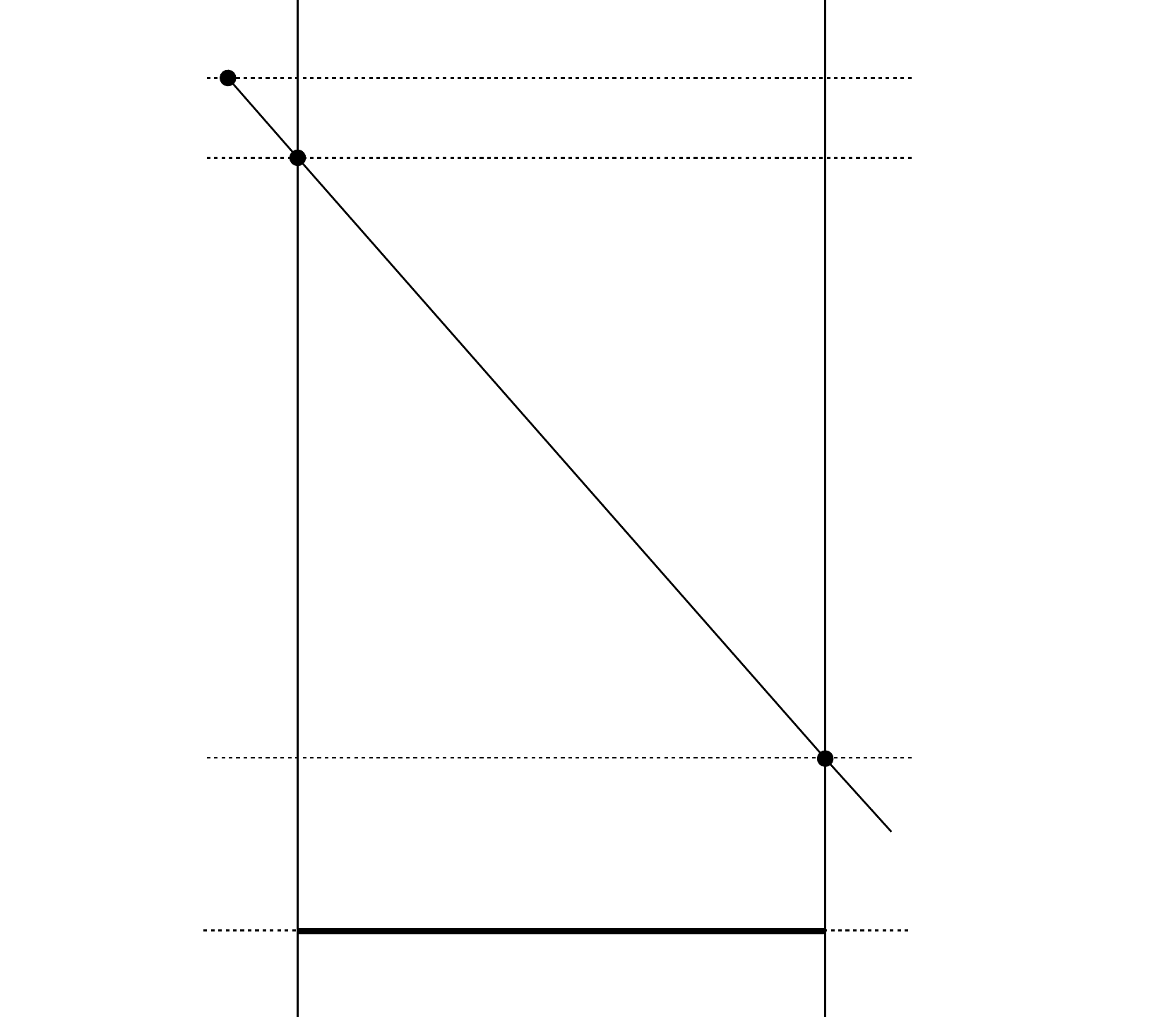
\end{minipage}
\begin{minipage}{0.5\linewidth}
\def\svgwidth{\textwidth}
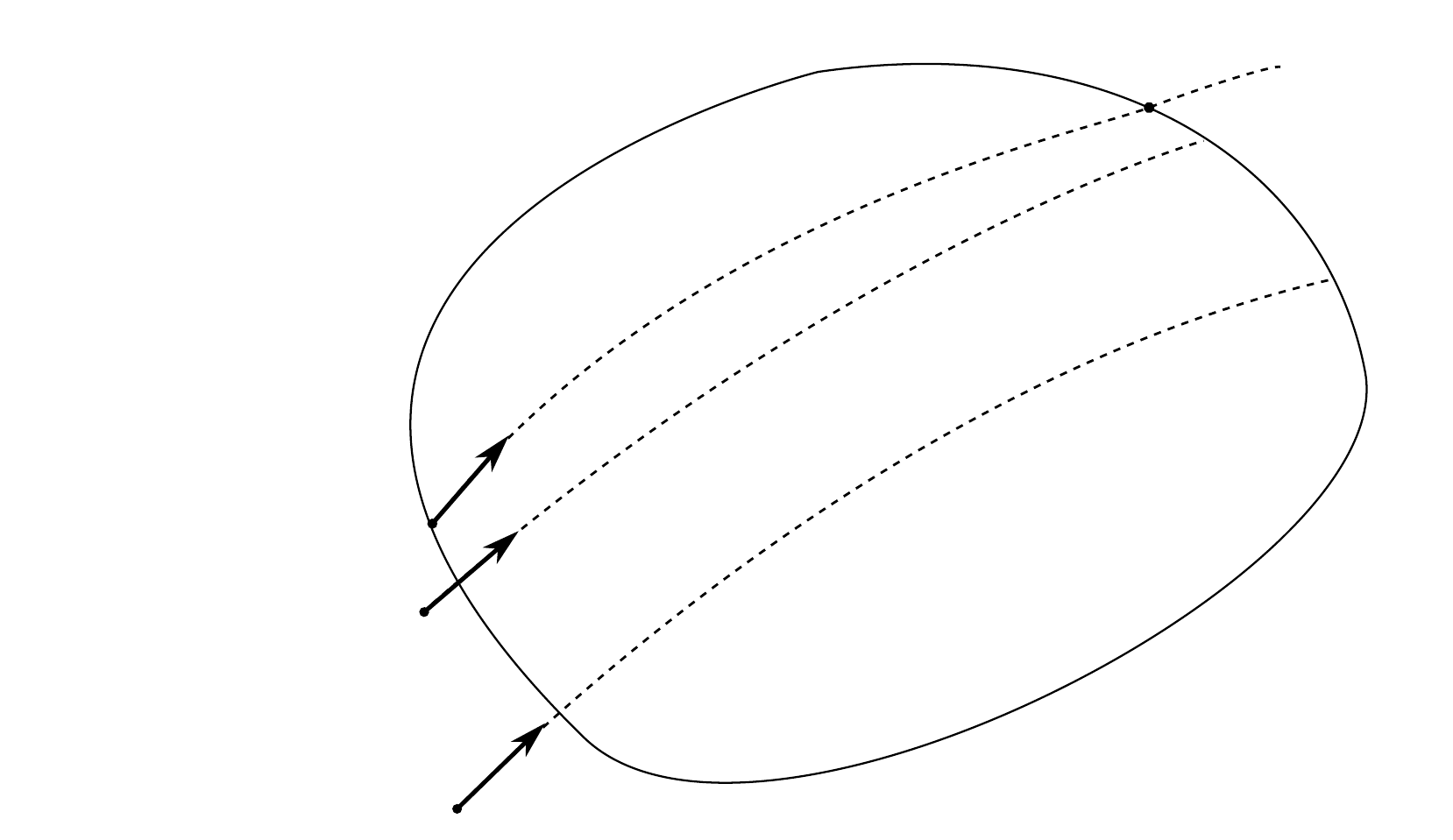
\end{minipage}
\caption{On left, trajectory of a Gaussian beam 
propagating along geodesic $\gamma(t) := \gamma(t; y_j, \eta_j)$
backwards on time interval $(T_0, t_j)$.
If $S(y_j, \eta_j, t_j) \ne 0$, then 
there is a point source at $\gamma(t_j)$.
On right, a sequnce $(y_j, \eta_j) \in S\Omega$ 
converging to $(x, \xi) \in \p_- S M$
and trajectories of the corresponding geodesics.
}\end{figure}

\begin{theorem}
\label{thm:determination_of_tau}
Let $(x, \xi) \in \p_- S M$ and denote by 
$J(x,\xi)$ the set of sequences  $((t_j; y_j, \eta_j))_{j=1}^\infty \subset (0,\infty) \times S\Omega$
for which
\begin{equation*}
\lim_{j \to \infty} (y_j, \eta_j) = (x,\xi), 
\quad \lim_{j \to \infty} t_j \in (0, \infty),
\quad S(y_j, \eta_j, t_j) \ne 0.
\end{equation*}
The function $S : S \Omega \times (0,\infty) \to \C$ 
determines $\tau : \p_- SM \to (0, \infty]$ by the formula
\begin{equation*}
\tau(x, \xi) = \inf\{\lim_{j \to \infty} t_j;\ ( (t_j; y_j, \eta_j))_{j=1}^\infty \in J(x,\xi)\
\text{for some $( (y_j, \eta_j))_{j=1}^\infty \subset S \Omega$} \}.
\end{equation*}
%where we define $\inf \emptyset = \infty$.

Moreover, if $\tau(x, \xi) < \infty$, then there is a sequence 
$((t_j; y_j, \eta_j))_{j=1}^\infty \in J(x,\xi)$
satisfying
\begin{equation*}
\tau(x,\xi) = \lim_{j \to \infty} t_j.
\end{equation*}
\end{theorem}
\begin{proof}
Let $(x, \xi) \in \p_- SM$ and $((t_j; y_j, \eta_j))_{j=1}^\infty \in J(x,\xi)$.
Let us show, that $\tau(x,\xi) \le \lim_{j \to \infty} t_j$.
By Lemma \ref{lem:S_is_enough}, $\tau_j := \tau(y_j, \eta_j) \to 0$ as $j \to \infty$.
We define
\begin{align*}
\tilde y_j := \gamma(\tau_j; y_j, \eta_j), \quad 
\xi_j := \dot \gamma(\tau_j; y_j, \eta_j).
\end{align*}
As $S(y_j, \eta_j, t_j) \ne 0$, we have
\begin{equation*}
\gamma(t_j - \tau_j; \tilde y_j, \xi_j) = \gamma(t_j; y_j, \eta_j) \in \p M.
\end{equation*}
As $\lim_{j \to \infty} t_j > 0$ and $\lim_{j \to \infty} \tau_j = 0$,
we have $t_j - \tau_j > 0$ for large $j$.
Thus $\tau(\tilde y_j, \xi_j) \le t_j - \tau_j$ for large $j$.
Moreover, 
\begin{equation*}
\lim_{j \to \infty} (\tilde y_j, \xi_j) = (\gamma(0; x, \xi), \dot \gamma(0; x, \xi)) = (x,\xi).
\end{equation*}
In particular, $(\tilde y_j, \xi_j) \in \p_- SM$ for large $j$. 
Hence Lemma \ref{lem:semicontinuity_of_tau} gives
\begin{equation*}
\lim_{j \to \infty} t_j 
= \lim_{j \to \infty} (t_j - \tau_j) 
\ge \liminf_{j \to \infty} \tau(\tilde y_j, \xi_j) \ge \tau(x,\xi).
\end{equation*}
In particular, we have proved the claim in the case $\tau(x, \xi) = \infty$.

Let us assume that $\tau(x, \xi) < \infty$.
It is enough to show, that there is a sequence 
$((t_j; y_j, \eta_j))_{j=1}^\infty \in J(x,\xi)$
satisfying $\tau(x,\xi) = \lim_{j \to \infty} t_j$.
We denote 
\begin{equation*}
t_0 := \tau(x,\xi), \quad
z := \gamma(t_0; x, \xi), 
\quad \zeta := -\dot \gamma(t_0; x, \xi).
\end{equation*}
We have
\begin{equation*}
(x, \xi) = (\gamma(t_0; z, \zeta), -\dot \gamma(t_0; z, \zeta)).
\end{equation*}
As $(x, \xi) \in \p_- S M$, 
Remark \ref{rem:transverse_intersection} gives
neighborhoods $I$ and $U$ of $t_0$ and $(z,\zeta)$ and a map $\ell : U \to I$ 
of boundary intersection times. 
After choosing local coordinates around $z$ we may define
\begin{equation*}
(y_j, \eta_j) := (\gamma(t_j; x_{k_j}, \zeta), -\dot \gamma(t_j; x_{k_j}, \zeta)),
\end{equation*}
where $(x_{k_j})_{j=1}^\infty \subset U$ is a subsequence
of the dense sequence of source points in (\ref{A formula})
satisfying $\lim_{j \to \infty} x_{k_j} = z$ and
$(t_j)_{j=1}^\infty \subset I$ satisfies 
\begin{equation*}
t_j > \ell(x_{k_j}, \zeta), \quad \lim_{j \to \infty} t_j = \ell(z, \zeta) = t_0.
\end{equation*}
Clearly $((t_j; y_j, \eta_j))_{j=1}^\infty \in J(x,\xi)$ and
\begin{equation*}
\lim_{j \to \infty} t_j = t_0 = \tau(x,\xi).
\end{equation*}
\end{proof}

\section{Determination of the scattering relation}
\label{sec:scattering_relation}

In the next theorem we consider pseudorandom noise $h_0(t,x)$ with 
coefficients 
\begin{equation*}
a_j = \lambda^{-\lambda^j},
\end{equation*}
with some $\lambda>1$ and make computations
``modulo an error in A'', where 
\begin{equation*}
A=\{-\lambda^j:\ j\in \N\}.
\end{equation*}
For this end, let $m_A(s)$ be 
the real number $r$ such that $s=r+a$ where $a\in A$ and 
$r$ has the smallest possible absolute value. 
In the case when both $r$ and $-r$
satisfy this condition, we choose the positive value.

%\begin{figure}
%\begin{center}
%%\includegraphics[width=0.9\textwidth]{sequence}
%\end{center}
%\caption{
%Trajectories of geodesics $\gamma(t; y_j, \eta_j)$
%of Lemma \ref{lem:hits}.
%}
%\end{figure}

\begin{lemma}
\label{lem:hits}
Let $(y_0,\eta_0)\in S\Omega$, 
$t_0 \in (0, T)$, and 
suppose that $S(y_0, \eta_0, t_0) \ne 0$. Then 
there is a sequence $((y_j, \eta_j))_{j=1}^\infty \subset S \Omega$
and $(t_j)_{j=1}^\infty \subset (0, T)$
such that
\begin{align} \label{eq:perturbed_hit}
&(y_j, \eta_j) \to (y_0, \eta_0),\ 
t_j \to t_0,\  
S(y_j, \eta_j, t_j) \to 0,\ \text{as $j \to \infty$},
\\\nonumber 
&S(y_j, \eta_j, t_j) \ne 0.
\end{align}

Suppose, moreover, that the coefficients of 
the pseudorandom noise $h_0$ are $a_j = \lambda^{-\lambda^j}$. 
Then for any sequences $((y_j, \eta_j))_{j=1}^\infty \subset T\R^n$
and $(t_j)_{j=1}^\infty \subset (0, T)$
satisfying (\ref{eq:perturbed_hit})
we have that
\begin{equation*}
\lim_{j \to \infty} m_A (\log_\lambda |S(y_j, \eta_j, t_j)|) 
= \log_\lambda|u_0(t_0, \gamma(t_0); y_0, \eta_0) |g|^{1/2}(\gamma(t_0))|,
\end{equation*}
where $\gamma(t) = \gamma(t; y_0, \eta_0)$
and $u_0$ is defined as in (\ref{rem:gaussian_beam_continuity_wrt_geodesic}).
\end{lemma}
\begin{proof}
We will use notation 
\begin{align*}
&\gamma_j(t) := \gamma(t; y_j, \eta_j), \quad
z_j := \gamma_j(t_j), \quad
S_j := S(y_j, \eta_j, t_j), 
\\&\beta_j := |u_0(t_j, z_j; y_j, \eta_j) |g|^{1/2}(z_j)|.
\end{align*}

As $S_0 \ne 0$, we have that $z_0 = x_j$ 
for some $j = 1, 2, \dots$.
By continuity of the geodesic flow and 
and density of $(x_j)_{j=1}^\infty \subset \p M$,
there exists a subsequence 
$(x_{k_j})_{j=1}^\infty \subset (x_j)_{j=1}^\infty$ 
and sequences $((y_j, \eta_j))_{j=1}^\infty \subset T\R^n$
and $(t_j)_{j=1}^\infty \subset (0, T)$
such that
\begin{equation*}
x_{k_j} \to z_0, \quad (y_j, \eta_j) \to (y_0, \eta_0), 
\quad t_j \to t_0,\quad\hbox{as }j\to \infty
\end{equation*}
and $z_j = x_{k_j} \ne z_0$.
Then $|S_j| = |a_{k_j}| \beta_j \ne 0$.
%where we have used the notation of Remark \ref{rem:gaussian_beam_continuity_wrt_geodesic}.
As $x_{k_j} \ne z_0$ and $x_{k_j} \to z_0$,
we have that $k_j \to \infty$ and thus $a_{k_j} \to 0$.
By (\ref{rem:gaussian_beam_continuity_wrt_geodesic})
and continuity of the geodesic flow, it holds that
$\beta_j \to \beta_0 > 0$.
Hence $S_j \to 0$.

Next we use the assumption that $a_j = \lambda^{-\lambda^j}$.
Let $((y_j, \eta_j))_{j=1}^\infty \subset T\R^n$
and $(t_j)_{j=1}^\infty \subset (0, T)$
satisfy (\ref{eq:perturbed_hit}).
As $S_j \ne 0$ we have that $|S_j| = a_{k_j} \beta_j$
for some subsequence 
$(a_{k_j})_{j=1}^\infty \subset (a_j)_{j=1}^\infty$.
As $S_j \to 0$ we have that $a_{k_j} \to 0$.
Moreover, sequence $(\log_2 \beta_j)_{j=1}^\infty$ is bounded.
This boundedness together with 
$\log_\lambda a_{k_j} \in A$ and $\log_\lambda a_{k_j} \to -\infty$ 
yield 
\begin{equation*}
m_A(\log_\lambda a_{k_j} + \log_\lambda \beta_j) = \log_\lambda \beta_j
\end{equation*}
for large $j \in \N$.
Hence,
\begin{equation*}
\lim_{j \to \infty} m_A (\log_\lambda |S_j|) 
= \lim_{j \to \infty} \log_\lambda \beta_j
= \log_\lambda \beta_0.
\end{equation*}
\end{proof}

\begin{theorem}
\label{thm:determination_of_intersection_point}
If the coefficients of 
the pseudorandom noise $h_0$ are $a_j = \lambda^{-\lambda^j}$,
then the functions $S : S \Omega \times (0,\infty) \to \C$
and $\tau: \p_- SM \to (0, \infty]$ 
determine $D(\Sigma)$ and
\begin{equation*}
\gamma(\tau(x, \xi); x, \xi), \quad (x, \xi) \in D(\Sigma).
\end{equation*}
\end{theorem}
\begin{proof}
Clearly $\tau$ on $\p_- SM$ determines $D(\Sigma)$.
Let $(x, \xi) \in D(\Sigma)$. 
By Theorem \ref{thm:determination_of_tau} we may choose
$((t_j; y_j, \eta_j))_{j=1}^\infty \in J(x, \xi)$ such that
$\lim_{j \to \infty} t_j = \tau(x, \xi)$.
As $S(y_j, \eta_j, t_j) \ne 0$,
$\gamma(t_j; y_j, \eta_j) = x_{k_j}$ for
some subsequence $(x_{k_j})_{j=1}^\infty$ of the sequence of source points.
By Lemma \ref{lem:hits} the function $S$ determines 
\begin{equation*}
\frac{|S(y_j, \eta_j, t_j)|}{
|u_0(t_j, x_{k_j}; y_j, \eta_j) |g|^{1/2}(x_{k_j})|}
= a_{k_j}.
\end{equation*}
As $a_j$, $j \in \Z_+$, are disjoint, 
this determines the index $k_j$ and thus also the point $x_{k_j}$. 
Moreover
\begin{equation*}
\lim_{j \to \infty} x_{k_j} 
= \lim_{j \to \infty} \gamma(t_j; y_j, \eta_j)
= \gamma(\tau(x,\xi); x, \xi).
\end{equation*}
\end{proof}

The following result follows from Remark 2.

\begin{lemma}
\label{lem:tau_when_intersection_not_tangential}
Let us denote by $X$ either $S \Omega$ or $\p_- S M$.
Let $(y_0, \eta_0) \in X$ satisfy 
\begin{equation*}
\tau(y_0, \eta_0) < \infty, 
\quad \dot \gamma(\tau(y_0, \eta_0); y_0, \eta_0) \notin T_z \p M, 
\end{equation*}
where $z = \gamma(\tau(y_0, \eta_0); y_0, \eta_0)$.
Then there is a neighborhood $V \subset X$ of $(y_0, \eta_0)$
such that $\tau = \ell$ in $V$,
where $\ell : U \to I$ is the map of boundary intersection times defined 
in Remark \ref{rem:transverse_intersection}
for neighborhoods $U \subset X$ and $I \subset \R$ of 
$(y_0, \eta_0)$ and $\tau(y_0, \eta_0)$.
In particular, $\tau$ is smooth in $V$.
\end{lemma}
%\begin{proof}
%As $\ell(y, \eta) = t$ is the unique solution of 
%\begin{equation*}
%\gamma(t; y, \eta) \in \p M, \quad t \in I,\ (y, \eta) \in U,
%\end{equation*}
%and $\gamma(\tau(y, \eta); y, \eta) \in \p M$ whenever $\tau(y, \eta) < \infty$,
%it is enough to show that $\tau(y,\eta) \in I$
%in a neighborhood of $(y_0, \eta_0)$.
%That is, it is enough to show that $\tau : X \to (0, \infty]$ 
%is continuous at $(y_0, \eta_0)$.
%Moreover, by Lemma \ref{lem:semicontinuity_of_tau} it is enough to show
%that $\tau$ is upper semi-continuous at $(y_0, \eta_0)$.

%We denote $\tau_0 := \tau(y_0, \eta_0)$.
%Let $\epsilon > 0$ satisfy $[\tau_0 - \epsilon, \tau_0 + \epsilon] \subset I$.
%By Remark \ref{rem:transverse_intersection}
%\begin{equation*}
%\gamma(\tau_0 - \epsilon; y_0, \eta_0) \quad \text{and}
%\quad \gamma(\tau_0 + \epsilon; y_0, \eta_0)
%\end{equation*}
%are on the different sides of $\p M$, that is one is in $\Omega$ and 
%the other is in $M$.
%Continuity of the geodesic flow gives a neighborhood $V_\epsilon \subset X$ 
%such that 
%\begin{equation*}
%\gamma(\tau_0 - \epsilon; y, \eta) \quad \text{and}
%\quad \gamma(\tau_0 + \epsilon; y, \eta)
%\end{equation*}
%are on different sides of $\p M$ for $(y, \eta) \in V_\epsilon$.
%Hence $\tau(y, \eta) \le \tau_0 + \epsilon$ for $(y, \eta) \in V_\epsilon$
%and $\tau$ is upper semi-continuous at $(y_0, \eta_0)$.
%\end{proof}

\begin{lemma}
\label{lem:parametric_transversality}
The set of $(x, \xi)$ such that $\gamma(\cdot; x, \xi)$ 
is transverse to $\p M$ is open and dense in 
\begin{equation*}
\p SM := \{(x, \xi) \in SM;\ x \in \p M\}.
\end{equation*}
\end{lemma}
\begin{proof}
As $\p_- SM \cup \p_+ SM$ is open and dense in $\p SM$, it is enough to show that 
the set of $(x, \xi)$ such that $\gamma(\cdot; x, \xi)$ 
is transverse to $\p M$ is open and dense in $\p_\pm SM$.
By the parametric transversality theorem, see \cite[Thm. 3.2.7]{Hirsch}, 
the claim follows from the fact that the evaluation map 
\begin{align*}
F^{ev} &: \p_\pm SM \times \R \to \R^n
\\F^{ev} &: (x, \xi, t) \mapsto \gamma(t; x, \xi)
\end{align*}
is transverse to $\p M$.
\end{proof}

\begin{lemma}
\label{lem:transversal_perturbation_and_tau}
Let $(x_0, \xi_0) \in D(\Sigma)$.
Then there is a sequence $( (x_j, \xi_j))_{j = 1}^\infty \subset D(\Sigma)$ such that
$\gamma(\cdot; x_j, \xi_j)$ is transverse to $\p M$ and 
\begin{equation*}
\lim_{j \to \infty} (x_j, \xi_j) = (x_0, \xi_0), 
\quad \lim_{j \to \infty} \tau(x_j, \xi_j)  = \tau(x_0, \xi_0).
\end{equation*}
\end{lemma}
\begin{proof}
We denote $\tau_0 := \tau(x_0, \xi_0)$ and
\begin{equation*}
(z_0, \zeta_0) := (\gamma(\tau_0; x_0, \xi_0), -\dot \gamma(\tau_0; x_0, \xi_0)).
\end{equation*}
Remark \ref{rem:transverse_intersection}
gives a map of boundary intersection times $\ell : U \to I$
for neighborhoods $U \subset S\R^n$ and $I \subset \R$ of 
$(z_0, \zeta_0)$ and $\tau_0$.
By Lemma \ref{lem:parametric_transversality} there is a sequence 
$( (z_j, \zeta_j))_{j=1}^\infty \subset SM \cap U$ converging to $(z_0, \zeta_0)$ such
that $\gamma(\cdot; z_j, \zeta_j)$ is transverse to $\p M$.

We define $t_j := \ell(z_j, \zeta_j)$ and 
\begin{equation*}
(x_j, \xi_j) := (\gamma(t_j; z_j, \zeta_j), - \dot \gamma(t_j; z_j, \zeta_j)).
\end{equation*}
Then $(x_j, \xi_j) \to (x_0, \xi_0)$ as $j \to \infty$.
In particular, there is $J \ge 1$ such that $(x_j, \xi_j) \in \p_- SM$ for 
$j \ge J$. By Lemma \ref{lem:semicontinuity_of_tau}
\begin{align*}
\tau(x_0, \xi_0) 
&\le \liminf_{j \to \infty} \tau(x_j, \xi_j) \le \limsup_{j \to \infty} \tau(x_j, \xi_j) 
\\&\le \lim_{j \to \infty} \ell(z_j, \zeta_j) = \ell(z_0, \zeta_0) = \tau(x_0, \xi_0).
\end{align*}
\end{proof}

\begin{lemma}
\label{lem:existence_of_variation_for_intersection_direction}
Let $(x_0, \xi_0) \in D(\Sigma)$ be such that $\gamma(\cdot; x_0, \xi_0)$
is transverse to $\p M$.
Then there is $(y_0, \eta_0) \in S \Omega$ lying on the geodesic $\gamma(\cdot; x_0, \xi_0)$
and a neighborhood $V \subset S_{y_0} \Omega$ of $\eta_0$ such that the following 
conditions hold.
\begin{itemize}
\item[(C1)] The map $\eta \mapsto \tau(y_0, \eta)$ is smooth $V \to (0, \infty)$.
\item[(C2)] The map 
\begin{equation}
\label{eq:variation_xeta}
(x(\eta), \xi(\eta)) := (\gamma(\tau(y_0, \eta); y_0, \eta), \dot \gamma(\tau(y_0, \eta); y_0, \eta))
\end{equation}
is smooth $V \to D(\Sigma)$ and $(x(\eta_0), \xi(\eta_0)) = (x_0, \xi_0)$.
\item[(C3)] The map
\begin{equation}
\label{eq:variation_teta}
\tilde \ell(\eta) := \tau(x(\eta), \xi(\eta)) + \tau(y_0, \eta)
\end{equation}
is smooth $V \to (0, \infty)$.
\item[(C4)] There is a neighborhood $W \subset \p M$ of 
$\gamma(\tau(x_0, \xi_0); x_0, \xi_0)$ such that
\begin{equation} \label{eq:variation_diffeo}
\eta \mapsto \gamma(\tau(x(\eta), \xi(\eta)); x(\eta), \xi(\eta))
\end{equation}
is a diffeomorphism $V \to W$.
\end{itemize}
\end{lemma}
\begin{proof}
We denote $\gamma(t) := \gamma(t; x_0, \xi_0)$ and 
$z_0 := \gamma(\tau(x_0, \xi_0))$.
By remark \ref{rem:transverse_intersection} $\gamma(-t) \in \Omega$
for small $t > 0$.
Moreover, the points that are  conjugate to $z_0$ along $\gamma$
are discrete on $\gamma$, see e.g. \cite{Jost}.

Thus there is $\tau_0 > 0$ such that 
\begin{equation*}
(y_0, \eta_0) := (\gamma(-\tau_0), \dot \gamma(-\tau_0))
\end{equation*}
is in $S \Omega$, $y_0$ is not conjugate to $z_0$ along $\gamma$,
$\tau(y_0,\eta_0) = \tau_0$ and 
\begin{equation*}
(\gamma(\tau_0; y_0, \eta_0), \dot \gamma(\tau_0; y_0, \eta_0)) = (x_0, \xi_0).
\end{equation*}

By Lemma \ref{lem:tau_when_intersection_not_tangential} there is
a neighborhood $V_0 \subset S_{y_0} \Omega$ of $\eta_0$ such that $\eta \mapsto \tau(y_0, \eta)$
is smooth in $V_0$.
Hence the function $\eta \mapsto (x(\eta), \xi(\eta))$ maps $\eta_0$ to $(x_0, \xi_0)$ and
is smooth in $V_0$.
Moreover, this smoothess, tranversality of $\gamma(\cdot, x_0, \xi_0)$
and Lemma \ref{lem:tau_when_intersection_not_tangential} 
imply that there is a neighborhood
$V_1 \subset V_0$ of $\eta_0$ such that $(x(\eta), \xi(\eta)) \in \p_- SM$
and $\eta \mapsto \tau(x(\eta), \xi(\eta))$ is smooth $V_1 \to (0,\infty)$.
In particular, $(x(\eta), \xi(\eta)) \in D(\Sigma)$ for all $\eta \in V_1$.
We have shown that $(y_0, \eta_0)$ and $V_1$ satisfy (C1)-(C3).

We have
\begin{equation}
\label{eq:group_property_across_boundary}
(\gamma(s; y_0, \eta), \dot \gamma(s; y_0, \eta))|_{s = t + \tau(y_0, \eta)} = 
(\gamma(t; x(\eta), \xi(\eta)), \dot \gamma(t; x(\eta), \xi(\eta))).
\end{equation}
In particular, $\gamma(\tilde \ell(\eta_0); y_0, \eta_0) = z_0$
and
\begin{equation*}
\gamma(\tilde \ell(\eta); y_0, \eta) = \gamma(\tau(x(\eta), \xi(\eta)); x(\eta), \xi(\eta)) \in \p M.
\end{equation*}

Moreover, as $y_0$ is not conjugate to $z_0$ along $\gamma$,
there are neighborhoods $V_2 \subset V_1$, $I_0 \subset (0, \infty)$ and $U_0 \subset \R^n$ of $\eta_0$, $\tilde \ell(\eta_0)$ and $z_0$
such that $(t, \eta) \mapsto \gamma(t; y_0, \eta)$
is a diffeomorphism $V_2 \times I_0 \to U_0$.

%By Lemma \ref{lem:tau_when_intersection_not_tangential},
%$t = \tau(x,\xi)$ is the unique solution of 
%\begin{equation*}
%\gamma(t; x, \xi) \in \p M
%\end{equation*} 
%in some neighborhoods %$I' \subset \R$ and $U' \subset \R^n$ of 
%of $\tau(x_0,\xi_0)$ and $(x_0, \xi_0)$. 
%Hence there is a neighborhood $V_3 \subset V_2$ of $\eta_0$ such that $\tilde \ell(\eta) \in I_0$ for all $\eta \in V_3$ and
%\begin{equation*}
%\end{equation*}
%is in $\p M$.
There is a neighborhood $V \subset V_2$ of $\eta_0$
such that $\tilde \ell(V) \subset I_0$.
The graph of $\eta \mapsto \tilde \ell(\eta)$ is an $(n-1)$ dimensional submanifold on $V \times I_0$.
Hence the diffeomorphism $(t, \eta) \mapsto \gamma(t; y_0, \eta)$
maps it onto a $(n-1)$ dimensional submanifold $W$ of $U_0$.
Moreover, $z_0 \in W$ and $W \subset \p M$. Thus $W$ is a neighborhood of $z_0$ in $\p M$.
\end{proof}

\begin{lemma}
\label{lem:variation_for_intersection_direction}
Let $(x_0, \xi_0) \in D(\Sigma)$ and $(y_0, \eta_0) \in S \Omega$
satisfy conditions (C1)-(C4) of Lemma \ref{lem:existence_of_variation_for_intersection_direction}
for neighborhoods $V \subset S_{y_0} \Omega$ and $W \subset \p M$
of $\eta_0$ and $z_0 := \gamma(\tau(x_0, \xi_0); x_0, \xi_0)$.
We denote by $F : W \to V$ the inverse map of (\ref{eq:variation_diffeo}).
Then 
\begin{equation}
\label{eq:tangential_component}
\grad_{\p M} (\tilde \ell \circ F)|_{z = z_0} = \dot \gamma_{z_0}^\top,
\end{equation}
where $\tilde \ell : V \to (0, \infty)$ is the function (\ref{eq:variation_teta})
and $\dot \gamma_{z_0}^\top$ is the orthogonal projection of $\dot \gamma(\tau(x_0, \xi_0); x_0, \xi_0)$
into $T_{z_0} \p M$.
\end{lemma}
\begin{proof}
Let $\sigma : (-\epsilon, \epsilon) \to W$ be a smooth curve such that $\sigma(0) = z_0$.
We define
\begin{equation*}
\Gamma : (-\epsilon, \epsilon) \times \R \to \R^n,
\quad \Gamma(s, t) := \gamma(t; y_0, F(\sigma(s))).
\end{equation*}
We denote $\lambda := \tilde \ell \circ F \circ \sigma$ and $\tilde \ell_0 := \tilde \ell(\eta_0)$.
By equation (\ref{eq:group_property_across_boundary})
\begin{align*}
&\Gamma(s, \lambda(s)) 
= \gamma(\tau(x(\eta), \xi(\eta)); x(\eta), \xi(\eta))|_{\eta = F(\sigma(s))}
%= F^{-1}(F(\sigma(s))) 
= \sigma(s),
\\&(\p_t \Gamma)(0, \tilde \ell_0) = \dot \gamma(\tilde \ell_0; y_0, \eta_0) = \dot \gamma(\tau(x_0, \xi_0); x_0, \xi_0).
\end{align*}
Hence
\begin{equation*}
\dot \sigma(0) 
= \p_s \Gamma(s, \lambda(s))|_{s = 0} 
= (\p_s \Gamma)(0, \tilde \ell_0) + (\p_t \Gamma)(0, \tilde \ell_0) \lambda'(0).
\end{equation*}
The curve $t \mapsto \Gamma(s, t)$ is a unit speed geodesic for all $s \in (-\epsilon, \epsilon)$.
Hence
\begin{align}
\label{eq:sigma_gamma_inner_product}
(\dot \sigma(0), \dot \gamma(\tau(x_0, \xi_0); x_0, \xi_0))_g
&= \l(( \p_s \Gamma, \p_t \Gamma)_g + \lambda'(0) ( \p_t \Gamma, \p_t \Gamma)_g \r)|_{s = 0, t = \tilde \ell_0}
\\\nonumber&= ( \p_s \Gamma, \p_t \Gamma)_g|_{s = 0, t = \tilde \ell_0} + \lambda'(0).
\end{align}

We define
\def\L{\mathcal L}
\begin{equation*}
\L(s,l) := \int_0^l |\p_t \Gamma(s, t) |_g dt, \quad (s, l) \in (-\epsilon, \epsilon) \times (0, \infty).
\end{equation*}
Then $\L(s, l)$, $s \in (-\epsilon, \epsilon)$ is the length of a unit speed geodesic on the interval $[0, l]$.
Thus $\L(s, l) = l$ for all $s \in (-\epsilon, \epsilon)$.
We may derive an expression for $\p_s \L(s, l)|_{s = 0}$ 
as in \cite[Prop. 6.5]{Lee}
\begin{equation*}
\p_s \L(s, l)|_{s = 0} = \int_0^l (D_t \p_s \Gamma, \p_t \Gamma)_g dt|_{s = 0}.
\end{equation*}
As $t \mapsto \Gamma(s, t)$ is a geodesic, $D_t \p_t \Gamma(s, t) = 0$ and thus
\begin{equation*}
\p_t (\p_s \Gamma, \p_t \Gamma)_g = (D_t \p_s \Gamma, \p_t \Gamma)_g.
\end{equation*}
Moreover, $\Gamma(s, 0) = y_0$ for all $s \in (-\epsilon, \epsilon)$ and thus $\p_s \Gamma(s, 0) = 0$.
Hence
\begin{equation*}
0 = \p_s \L(s, l)|_{s = 0} = \int_0^l \p_t (\p_s \Gamma, \p_t \Gamma)_g dt|_{s = 0}
= (\p_s \Gamma, \p_t \Gamma)_g|_{s = 0, t = l},
\quad l \in (0, \infty).
\end{equation*}

By (\ref{eq:sigma_gamma_inner_product}) we have 
\begin{align*}
(\dot \sigma(0), \gamma_{z_0}^\top)_g 
&= (\dot \sigma(0), \dot \gamma(\tau(x_0, \xi_0); x_0, \xi_0))_g 
\\&= \lambda'(0) = \pair{d(\tilde \ell \circ F)|_{z = z_0}, \dot \sigma(0)}_{T_{z_0}^* \p M \times T_{z_0} \p M}
\\&= (\dot \sigma(0), \grad_{\p M} (\tilde \ell \circ F)|_{z = z_0})_g,
\end{align*}
for all smooth curves $\sigma$ in $W$ such that $\sigma(0) = z_0$,
which proves the claim.
\end{proof}

\begin{theorem}
\label{thm:determination_of_intersection_direction}
The functions $\tau: \p_- SM \to (0, \infty]$ and 
\begin{equation*}
z : D(\Sigma) \to \p M, 
\quad z(x, \xi) := \gamma(\tau(x, \xi); x, \xi) 
\end{equation*}
together with the Riemannian manifold $(\Omega, g|_\Omega)$ determine
\begin{equation*}
\dot \gamma(\tau(x, \xi); x, \xi), \quad (x, \xi) \in D(\Sigma).
\end{equation*}
\end{theorem}
\begin{proof}
The functions $\tau$ and $z$ on $D(\Sigma)$ determine the set $B$
of points $(x_0, \xi_0) \in D(\Sigma)$ such that the conditions (C1)-(C4) 
of Lemma \ref{lem:existence_of_variation_for_intersection_direction} hold
for some $(y_0, \eta_0) \in S \Omega$.

Let $(x_0, \xi_0) \in B$. 
We denote $\zeta_0 := \dot \gamma(\tau(x_0, \xi_0); x_0, \xi_0)$.
The map
\begin{equation*}
\eta \mapsto z(x(\eta), \xi(\eta))
\end{equation*}
determines its local inverse.
Hence $\tau$ and $z$ determine the
function $F$ of Lemma \ref{lem:variation_for_intersection_direction},
and thus they determine $\dot \gamma_{z_0}^\top$ by the formula (\ref{eq:tangential_component}).
As $\zeta_0$ is a unit vector
\begin{equation*}
\zeta_0 = \dot \gamma_{z_0}^\top + (1- |\dot \gamma_{z_0}^\top|^2)^{1/2} \nu_{z_0},
\end{equation*}
where $\nu_{z_0}$ the unit exterior normal vector of $\p M$.
Hence $\tau$ and $z$ determine $\zeta_0$ for all $(x_0, \xi_0) \in B$.

Let $(x_0, \xi_0) \in D(\Sigma)$.
By Lemmata \ref{lem:transversal_perturbation_and_tau} and \ref{lem:existence_of_variation_for_intersection_direction}
there is a sequence $((x_j, \xi_j))_{j=1}^\infty \subset B$ such that 
\begin{equation*}
\lim_{j \to \infty} (x_j, \xi_j) = (x_0, \xi_0), 
\quad \lim_{j \to \infty} \tau(x_j, \xi_j)  = \tau(x_0, \xi_0).
\end{equation*}
Moreover, the functions $\tau$ and $z$ determine the set of such sequences
and thus they determine 
\begin{equation*}
\lim_{j \to \infty} \dot \gamma(\tau(x_j, \xi_j); x_j, \xi_j) 
= \dot \gamma(\tau(x_0, \xi_0); x_0, \xi_0).
\end{equation*}
\end{proof}

Theorems \ref{thm:determination_of_tau}, \ref{thm:determination_of_intersection_point}
and \ref{thm:determination_of_intersection_direction} 
prove Theorem \ref{main_thm} formulated in the introduction.

\medskip
{\bf Acknowledgements.}
The authors were partly supported by Finnish Centre of Excellence in Inverse Problems Research,
Academy of Finland COE 213476.
ML was partly supported also by Mathematical Sciences Research Institute.
LO was partly supported also by Finnish Graduate School in Computational Sciences.

\end{document}

%% file: distance.pdf_tex
%% Creator: Inkscape 0.48.0, www.inkscape.org
%% Accompanies image file 'distance.pdf' (pdf, eps, ps)

\begingroup
  \makeatletter
  \providecommand\color[2][]{%
    \errmessage{(Inkscape) Color is used for the text in Inkscape, but the package 'color.sty' is not loaded}
    \renewcommand\color[2][]{}%
  }
  \providecommand\transparent[1]{%
    \errmessage{(Inkscape) Transparency is used (non-zero) for the text in Inkscape, but the package 'transparent.sty' is not loaded}
    \renewcommand\transparent[1]{}%
  }
  \providecommand\rotatebox[2]{#2}
  \ifx\svgwidth\undefined
    \setlength{\unitlength}{476.15045776pt}
  \else
    \setlength{\unitlength}{\svgwidth}
  \fi
  \global\let\svgwidth\undefined
  \makeatother
  \begin{picture}(1,0.88699365)%
    \put(0,0){\includegraphics[width=\unitlength]{distance.pdf}}%
    \put(0.79702122,0.82258714){\color[rgb]{0,0,0}\makebox(0,0)[lb]{\smash{$t=t_j$}}}%
    \put(0.79702122,0.71629729){\color[rgb]{0,0,0}\makebox(0,0)[lb]{\smash{$t=\tau(x,\xi)$}}}%
    \put(0.79702122,0.22727848){\color[rgb]{0,0,0}\makebox(0,0)[lb]{\smash{$t=0$}}}%
    \put(-0.00119776,0.85847354){\color[rgb]{0,0,0}\makebox(0,0)[lb]{\smash{$(y_j,t_j)$}}}%
    \put(0.34871273,0.68168263){\color[rgb]{0,0,0}\makebox(0,0)[lb]{\smash{$(x,\tau(x,\xi))$}}}%
    \put(0.3894889,0.16078359){\color[rgb]{0,0,0}\makebox(0,0)[lb]{\smash{$(\gamma(t_j),0)$}}}%
    \put(0.31551367,0.00346201){\color[rgb]{0,0,0}\makebox(0,0)[lb]{\smash{$M$}}}%
    \put(0.79581685,0.07827536){\color[rgb]{0,0,0}\makebox(0,0)[lb]{\smash{$t=T_0$}}}%
  \end{picture}%
\endgroup

%% file: sequence.pdf_tex
%% Creator: Inkscape 0.48.0, www.inkscape.org
%% Accompanies image file 'sequence.pdf' (pdf, eps, ps)

\begingroup
  \makeatletter
  \providecommand\color[2][]{%
    \errmessage{(Inkscape) Color is used for the text in Inkscape, but the package 'color.sty' is not loaded}
    \renewcommand\color[2][]{}%
  }
  \providecommand\transparent[1]{%
    \errmessage{(Inkscape) Transparency is used (non-zero) for the text in Inkscape, but the package 'transparent.sty' is not loaded}
    \renewcommand\transparent[1]{}%
  }
  \providecommand\rotatebox[2]{#2}
  \ifx\svgwidth\undefined
    \setlength{\unitlength}{477.31373291pt}
  \else
    \setlength{\unitlength}{\svgwidth}
  \fi
  \global\let\svgwidth\undefined
  \makeatother
  \begin{picture}(1,0.5781913)%
    \put(0,0){\includegraphics[width=\unitlength]{sequence.pdf}}%
    \put(0.1467071,0.19884925){\color[rgb]{0,0,0}\makebox(0,0)[lb]{\smash{$(x,\xi)$}}}%
    \put(0.60463962,0.55198671){\color[rgb]{0,0,0}\makebox(0,0)[lb]{\smash{$\gamma(\tau(x,\xi); x,\xi)$}}}%
    \put(0.66143452,0.17036899){\color[rgb]{0,0,0}\makebox(0,0)[lb]{\smash{$M$}}}%
    \put(-0.00119484,0.11869461){\color[rgb]{0,0,0}\makebox(0,0)[lb]{\smash{$(y_{j+1},\eta_{j+1})$}}}%
    \put(0.1259257,0.00705445){\color[rgb]{0,0,0}\makebox(0,0)[lb]{\smash{$(y_j,\eta_j)$}}}%
  \end{picture}%
\endgroup